\documentclass{scrartcl}
\usepackage{cite}
\usepackage{eucal}
\usepackage{kotex}
\usepackage{lineno,hyperref}
\usepackage{epstopdf}
\usepackage{pifont}
\usepackage{latexsym}
\usepackage{graphicx}
\usepackage{float}
\usepackage[dvips]{xy}
\xyoption{all}
\usepackage{ifpdf}
\usepackage{multirow}
\usepackage{amssymb, amsthm, amsmath}
\usepackage{hhline}
\usepackage{centernot}
\usepackage{verbatim, comment, color}
\usepackage{enumerate}
\usepackage{setspace}
\usepackage{lipsum}
\usepackage{subcaption}
\usepackage{bbm}

\usepackage{mathtools}

\DeclareOldFontCommand{\bf}{\normalfont\bfseries}{\mathbf}

\newtheorem{theorem}{Theorem} 
\newtheorem{proposition}[theorem]{Proposition}
\newtheorem{lemma}[theorem]{Lemma}
\newtheorem{remark}[theorem]{Remark}

\newcommand{\be}{\begin{equation}}

\newcommand{\ee}{\end{equation}}
\newcommand{\bea}{\begin{eqnarray}}
\newcommand{\eea}{\end{eqnarray}}
\newcommand{\barr}{\begin{array}}
\newcommand{\earr}{\end{array}}
\newcommand{\bn}{\begin{enumerate}}
\newcommand{\en}{\end{enumerate}}
\newcommand{\bi}{\begin{itemize}}
\newcommand{\ei}{\end{itemize}}
\newcommand{\bbbm}{\begin{pmatrix}}
\newcommand{\eeem}{\end{pmatrix}}
\newcommand{\f}{\frac}

\newcommand{\R}{{\mathbb R}}
\newcommand{\E}{{\mathbb E}}

\newcommand{\ga}{\gamma}

\newcommand{\de}{\delta}

\newcommand{\ep}{\varepsilon}

\newcommand{\tht}{\theta}
\newcommand{\ignore}[1]{}{}

\newcommand{\nn}{\nonumber}

\newcommand{\p}{{\partial}}

\newcommand{\lt}{\left}
\newcommand{\rt}{\right}

\newcommand{\ot}{\otimes}

\newcommand{\Prob}{{\mathcal P}}
\newcommand{{\QED}}{{\hfill QED} \smallskip}

\newcommand{\et}{\exp}

\newcommand{\grad}{\nabla}
\newcommand{\foe}{\f{1}{\eta}}

\newcommand{\TabFive}[1]{ %
\begin{tabular}{@{}c@{\hspace{2mm}}c@{\hspace{2mm}}c@{\hspace{2mm}}c@{\hspace{2mm}}c@{}}
 #1
\end{tabular}
}

\definecolor{darkspringgreen}{rgb}{0.09, 0.45, 0.27} 
\definecolor{darkgray}{rgb}{0.66, 0.66, 0.66}

\numberwithin{equation}{section}
\numberwithin{theorem}{section}

\begin{document}

\title{An ordinary differential equation for entropic optimal transport and its linearly constrained variants}

\author{Joshua Zoen-Git Hiew,\thanks{Department of Mathematical and Statistical Sciences, 632 CAB, University of Alberta, Edmonton, Alberta, Canada, T6G 2G1. email: joshuazo@ualberta.ca}\; 
Luca Nenna\thanks{Universit\'e Paris-Saclay, CNRS, Laboratoire de math\'ematiques d'Orsay, ParMA, Inria Saclay, 91405, Orsay, France. 
email: luca.nenna@universite-paris-saclay.fr}\;  and  Brendan Pass \thanks{Department of Mathematical and Statistical Sciences, 632 CAB, University of Alberta, Edmonton, Alberta, Canada, T6G 2G1.
email: pass@ualberta.ca}}
\maketitle
\begin{abstract} 
We characterize the solution to the entropically regularized optimal transport problem by a well-posed ordinary differential equation (ODE).  Our approach works for discrete marginals and general cost functions, and in addition to two marginal problems, applies to multi-marginal problems and those with additional linear constraints.  Solving the ODE gives a new numerical method to solve the optimal transport problem, which has the advantage of yielding the solution for all intermediate values of the ODE parameter (which is equivalent to the usual regularization parameter). We illustrate this method with several numerical simulations. The formulation of the ODE also allows one to compute derivatives of the optimal cost when the ODE parameter is $0$, corresponding to the fully regularized limit problem in which only the entropy is minimized. 
\end{abstract}

\vskip\baselineskip\noindent


\section{Introduction}

Given probability measures $\mu^1$ and $\mu^2$ on domains $X^1,X^2 \subseteq \mathbb{R}^d$, respectively, and a cost function $c: X^1 \times X^2 \rightarrow \mathbb{R}$, the Monge-Kantorovich optimal transport problem is to minimize
\begin{equation}\label{eqn: two marginal unconstrained problem}
\int_{X^1 \times X^2}c(x^1,x^2) d\gamma(x^1,x^2)
\end{equation}
over the set $\Pi(\mu^1,\mu^2)$ of joint measures on $X^1  \times X^2$ having the $\mu^i$ as marginals.  This simply stated problem has grown exponentially in recent years.  It has an extremely wide variety of applications and has spawned many extensions and variants; see the books of Villani \cite{Villani03, Villani09} and Santambrogio \cite{Santambrogio15} for detailed surveys.

A particularly popular modification of optimal transport is \emph{entropic regularization}, which arises when one penalizes the \emph{entropy}, 
$$
H_{\mu^1 \otimes \mu^2}(\gamma) =\int_{X^1 \times X^2} \frac{d\gamma}{d(\mu^1 \otimes \mu^2)} \log\Big(\frac{d\gamma}{d(\mu^1 \otimes \mu^2)} \Big)d(\mu^1 \otimes \mu^2),
$$
($H$ is taken to be $+\infty$ when $\gamma$ is not absolutely continuous with respect to $\mu^1 \otimes \mu^2$) in the transport functional, and so minimizes  $\int_{X^1 \times X^2}c(x^1,x^2) d\gamma(x^1,x^2) +\eta H_{\mu^1 \otimes \mu^2}(\gamma)$ for some small $\eta >0$.  As the regularization parameter $\eta$ goes to $0$, solutions are well known to tend to solutions of the unregularized problem.  Much of the early motivation for entropic regularization was computational; in contrast to the linear unregularized problem, the optimization problem is now strictly convex, and the unique solution can be efficiently computed via the celebrated Sinkhorn algorithm \cite{cuturi2013, benamouetalentropic}.  The theoretical properties of the curve  $\gamma_\eta$ of solutions, as a function of the regularization parameter $\eta$, including its rate of convergence and Taylor expansion about $\eta =0$, are also topics of current research interest \cite{leonard2013, Pal24, BerntonGhosalNutz22, conforti2021, CarlierPegonTamanini23, NennaPegon23}.

In this paper, we show that, for discrete marginals\footnote{General marginals will be dealt with in a separate, forthcoming paper.}, solutions to the regularized problem can be characterized via a well-posed \emph{ordinary differential equation} (ODE); see Theorem \ref{wellposeODE}.  In fact, we do this for problems that are considerably more general than \eqref{eqn: two marginal unconstrained problem} in two ways: we allow \emph{multiple marginals}, as reviewed in \cite{pass2015, DiMarinoGerolinNenna17} and \emph{additional linear constraints} as formulated in \cite{zaev15}.  This general framework includes many applications in addition to the classical, unconstrained two marginal problem, including multi-marginal problems arising in the evaluation of Wasserstein barycenters \cite{aguehcarlier2011} and in density functional theory \cite{ButtazzoDePascaleGoriGiorgi2012, cotarFrieseckeKluppelberg2013}, for example, as well as constrained problems such as martingale optimal transport, vectorial martingale optimal transport, adapted optimal transport and relaxed weak optimal transport, among others \cite{acciaio2021, backhoff2020, beiglbock2021, beiglbock2016, lassalle2015, hiew2023}.  The precise general problem we address, after regularizing as above, is formulated in \eqref{EOT_linCons_prim} below.

The development of our ODE results in two distinct contributions.  First, it yields a new numerical method to compute solutions to entropically regularized optimal transport and its linearly constrained variants, by solving the corresponding ODE.  This is similar in spirit to the method proposed by two of the present authors in \cite{nenna2022}, but applies to much more general cost functions. Note that the work in \cite{nenna2022} applies only to multi-marginal problems with pairwise cost functions, and that the initial condition for the ODE derived there requires solving $n-1$ two marginal optimal transport problems, whereas the initial condition in our formulation here has a simple closed form solution (at least for problems without extra constraints).  In particular, the approach in \cite{nenna2022} does not apply to the two marginal problem, and additional linear constraints were not addressed there.  In fact, we would like to emphasize that the ODE approach seems very promising for computing other curves of measures.  To illustrate this point, we show in Section \ref{subsect: barycenters and geodesics} below that a slight modification of our general framework here can be used to compute (and provide a new perspective on) Wasserstein geodesics (also known as displacement interpolants \cite{McCann1994}) and  barycenters \cite{aguehcarlier2011}.

 We verify, through several numerical simulations, the feasability of the  method presented here to compute solutions to a variety of problems. A key computational advantage of our approach is that it yields the \emph{entire curve} $\gamma(\ep)$ of solutions, which interpolates between product measure and optimal transport as the ODE variable $\ep$ varies ($\ep$ is essentially equivalent to  $1/\eta$ -- see \eqref{EOT_linCons_ep_prim} below for the precise definition).  This curve is of potential independent interest, due to current research on the behaviour of $\gamma$ as a function of  $\eta$; constructing it using the Sinkhorn method would require a separate Sinkhorn calculation for each $\ep$ and would therefore be much less efficient. 

A second contribution arising from the formulation of the ODE is that it allows one to calculate derivatives of the optimal cost around $\ep =0$ (the fully regularized limit, corresponding to $\eta \rightarrow \infty$).  Taylor expansions of the cost around the fully unregularized, optimal transport limit ($\ep \rightarrow \infty$, or $\eta\rightarrow 0$) have been an active topic of recent research, whereas  expansions around the fully regularized limit have received relatively limited attention (see \cite{conforti2021} for a first order expansion).  The ODE formulation easily allows one to calculate derivatives (albeit with discrete marginals) in closed form.  We illustrate this by obtaining a formula for the second derivative $C''(0)$ of the optimal cost for unconstrained, two marginal, optimal transport.

Our work here has another, modest and somewhat incidental, consequence.  Aside from developing a general ODE method to solve linearly constrained optimal transport problems, we also provide a general framework for applying the Sinkhorn algorithm to this class of problems, unifying and extending previous work on particular cases \cite{demarch2018, eckstein2022}.  The key to this are two general lemmas on the strong convexity of the objective functional after an appropriate reduction in the variables (Lemmas \ref{lem: reduction of A} and \ref{ThmPhiConv} below), enabling the implementation of the general Sinkhorn algorithm (also known as the block coordinate descent method), found, for instance, in \cite{OrtegaRheinboldt00} -- see Proposition \ref{lem: gen SK} below.  This fact also justifies the use of the Sinkhorn algorithm in Section \ref{numericSec} to generate reference solutions (especially for the multi-period martingale problem in Section \ref{subsect: multi-period martingales}, as to the best of our knowledge, a version of the Sinkhorn algorithm has not been developed for this type of problem before).

This paper is structured as follows. In Section 2, we introduce optimal transport with extra linear constraints and the entropic regularization of the corresponding linear programming problem. In this setting, we derive the ODE and prove its well-posedness in Section 3. In Section 4, we derive formulas for higher-order derivatives of the optimal value of unconstrained optimal transport as a function of $\ep$. Finally, in Section 5, we demonstrate several numerical examples based on our ODE methodology and compare them with the traditional Sinkhorn algorithm. 

\section{Optimal Transport with Extra Linear constraints}

\subsection{The basic problem}

Let $\mu^1,...,\mu^n$ be probability measures\footnote{From here on, the superscript index will denote various marginals or spaces, or different linear constraints, while the subscript index will represent the elements within certain marginals or spaces.} on respective bounded domains $X^1,...,X^n$ in $\R^d$. Set $X \coloneqq X^1 \times ... \times X^n$. Let $\mathcal{P}(X)$ be the set of probability measures on $X$ and $\Pi(\mu^1,...,\mu^n)$ be the set of couplings of $\mu^1,...,\mu^n$; that is, measures $\gamma \in \mathcal{P}(X)$ whose marginals are the $\mu^i$. Let $Q \subset C_b(X)$ be a subspace of bounded continuous functions.

If $c: X \to \R$ is the cost function, then the optimal transport problem with extra linear constraints is the following optimization problem:
\be\label{eqn: unregularized primal}
    \inf_{\ga \in \Pi^Q(\mu^1,...,\mu^n)} \int_X c d\ga,
\ee
where $\Pi^Q(\mu^1,...,\mu^n) \subset \Pi(\mu^1,...,\mu^n)$ is the set of couplings such that $\int q d \ga = 0$ for all $q \in Q$.

\begin{remark}
The marginal constraints themselves can be represented within the linear constraint context; letting $Q^i = \{f^i - \int f^i d\mu^i | f^i \in  C_b(X^i) \}$, the subspace $Q \coloneqq \oplus Q^i$ of functions enforces the marginal constraints.  We choose here to distinguish these constraints, denoting them by $\Pi(\mu^1,...,\mu^n)$, to emphasize our focus on optimal transportation.  This is contrast to the notational convention in  \cite{beiglbock2014}, where they include the marginal constraints within the subspace $Q$.
\end{remark}

\begin{remark}
    Note that if $Q = \{0\}$, then the problem reduces to unconstrained, multi-marginal optimal transport.
\end{remark}

It is easy to see that the existence of an optimizer is equivalent to the non-emptiness of the feasible set, provided that the function $c$ satisfies certain mild conditions\cite{zaev15}.

The corresponding dual problem is:
\be\label{eqn: unregularized dual}
    \sup_{(\psi, q) \in \Psi(Q, c)}  \sum_{i=1}^n \int_{X_i} \psi^i d \mu^i, 
\ee
where $\Psi(Q, c) \coloneqq \{ (\psi, q) | \sum_{i=1}^n \psi^i + q \leq c ,\psi = (\psi^1,...,\psi^n), \psi^i \in L^1(\mu^i), q \in Q\}$

As proven in \cite{zaev15}, no duality gap exists between the primal \eqref{eqn: unregularized primal} and dual \eqref{eqn: unregularized dual} problems; that is, the optimal values in the two problems coincide. 

For a given $\eta >0$, we are interested here in the following entropically regularized version of \eqref{eqn: unregularized primal}:

\be\label{EOT_linCons_prim}
\boxed{    \inf_{\ga \in \Pi^Q(\mu^1,...,\mu^n)} \int_X c d\ga + \eta H_{\ot_{i=1}^n \mu^i}(\ga),}
\ee
where $H_{\ot_{i=1}^n \mu^i}(\ga)$ denotes the relative entropy of $\ga$ with respect to the product measure $\ot_{i=1}^n \mu^i$ of the corresponding marginals, defined as:
\be
    H_{\ot_{i=1}^n \mu^i}(\ga) = \int_X \f{d \ga}{d (\ot_{i=1}^n \mu^i)} \log \lt( \f{d \ga}{d (\ot_{i=1}^n \mu^i)} \rt) d(\ot_{i=1}^n \mu^i) \nn,
\ee
when $\ga$ is absolutely continuous with respect to the product measure, and $+\infty$ otherwise.  This regularization has been studied extensively in the classical optimal transport problem and many variants. The dual of \eqref{EOT_linCons_prim} is the unconstrained optimization problem: 

\be\label{EOT_linCons_dual}
    \boxed{\sup_{(\psi, q) \in \Psi(Q, c)}  \sum_{i=1}^n \int_{X^i} \psi^i d \mu^i - \eta \int_X \text{exp}\lt(\f{\sum_{i=1}^n \psi^i + q - c}{\eta}\rt) d (\ot_{i=1}^n \mu^i).}
\ee

As the objective function in \eqref{EOT_linCons_prim} is strictly convex with respect to $\gamma$ for $\eta > 0$, the optimization problem has a unique optimal solution, denoted by $\gamma_\eta$.  Similarly, as we show below (in the discrete case -- see Lemmas \ref{lem: reduction of A} and \ref{ThmPhiConv}), if we restrict to an appropriate subspace of $\Psi(Q,c)$ then the objective function in \eqref{EOT_linCons_dual} is strictly concave with respect to $(\psi, q)$, also resulting in a unique solution $(\psi_\eta, q_\eta)$ when $\eta > 0$.  It is well known in the unconstrained setting $Q=\{0\}$ that cluster points of solutions of $\gamma_\eta$ as $\eta \rightarrow 0$ solve the unregularized problem \eqref{eqn: unregularized primal}; furthermore, under strong additional assumptions, $\gamma_\eta$ converges to the unique measure $\gamma_0$ having minimal entropy \emph{among all minimizers of \eqref{eqn: unregularized primal}} (see Theorem 18 in \cite{Rodrigue2023} for a fairly general result in this direction).  These results carry over easily to the more general problem with nontrival $Q$; we provide a proof in Appendix \ref{AppPrimeConvProof} for the sake  of completeness.  We also show there that cluster points of $(\psi_\eta, q_\eta)$ solve \eqref{eqn: unregularized dual}, at least in the discrete setting that we specialize to below (for convergence in the unconstrained case, see, for example, \cite{NutzWiesel2022}). 
We are interested here in the curve $\eta\mapsto \gamma_\eta$, which interpolates between the optimal transport $\gamma_0$ at $\eta=0$ and the entropy minimizing element of $\Pi^Q(\mu^1,...,\mu^n)$ as $\eta \rightarrow \infty$.  In fact, it will be convenient for us to view this interpolation in a slightly different (but equivalent) way: we fix $\eta$ and introduce a second parameter $\ep \geq 0$ to obtain the problem 

\be\label{EOT_linCons_ep_prim}
    \inf_{\ga \in \Pi^Q(\mu^1,...,\mu^n)} \int_X \ep c d\ga + \eta H_{\ot_{i=1}^n \mu^i}(\ga).
\footnote{One could of course more simply set $\ep =1/\eta$ in the original problem \eqref{EOT_linCons_prim}, which corresponds to taking $\eta =1$ in \eqref{EOT_linCons_ep_prim}.  However, we find it more convenient to retain the parameter $\eta$ here was well.  In particular, below we will characterize solutions by an ODE in $\ep$.  Keeping $\eta$ will allow us to solve \eqref{EOT_linCons_prim} by solving the ODE up to $\ep =1$ for any value of $\eta$, instead of having to consider all values of $\ep$ from $0$ up to $\infty$.}\ee
For a fixed $\eta$, as $\ep \to 0$, this becomes equivalent to finding the couplings that satisfy the extra linear constraints while minimizing the relative entropy:
\be
    \inf_{\ga \in \Pi^Q(\mu^1,...,\mu^n)} H_{\ot_{i=1}^n \mu^i}(\ga). \nn
\ee
Note that in the unconstrained problem $Q = \{0\}$, a trivial solution arises—the product measure $\gamma \in \Pi^Q(\mu^1,...,\mu^n)$ itself.

Conversely, if we let $\ep \to \infty$, the problem reverts to one equivalent to the original unregularized optimal transport problem with extra linear constraints \eqref{eqn: unregularized primal}, while at $\ep=1$ we recover \eqref{EOT_linCons_prim}.  We are interested here in the minimmizer $\gamma(\ep)$ in \eqref{EOT_linCons_ep_prim} as a function of $\ep$, and will show in Section  \ref{sect: ODE} below that for discrete marginals it can be characterized by an ODE in $\ep$.

The dual of \eqref{EOT_linCons_ep_prim} is simply \eqref{EOT_linCons_dual} with the cost function $\ep c$ replacing $c$:
\be\label{EOT_linCons_ep_dual}
    \sup_{(\psi, q) \in \Psi(Q, c)}  \sum_{i=1}^n \int_{X_i} \psi^i d \mu^i - \eta \int_X \text{exp}\lt(\f{\sum_{i=1}^n \psi^i + q - \ep c}{\eta}\rt) d (\ot_{i=1}^n \mu^i).
\ee

\subsection{ The discrete regularized problem}\label{OTexLinCtrLP}
We now specialize to the discrete marginal setting.  Assume that each $\mu^i = \sum_{x_i \in X^i}\mu^{i}_{x^i}\delta_{x^i}$ is supported on a finite subset $X^i\subseteq \mathbb{R}^d$; in this case, $C_b(X)$ is finite dimensional. The constraint subspace $Q$ becomes finite dimensional with basis $\{q^j\}_{j=1}^K$ for some $K \geq 0$ ($K=0$ is unconstrained optimal transport). The cost function $c=(c_x)_{x\in X}$, constraint vectors $q^j=(q_x^j)_{x\in X}$ and product measure $\boldsymbol{\mu} = (\boldsymbol{\mu}_x)_{x \in X} = (\ot_{i=1}^n \mu^i_{x^i})_{x \in X}$ become vectors in $\mathbb{R}^m$ for $m=\Pi_{i=1}^n N^i =|X|$, where $N^i=|X^i|$. Problem \eqref{EOT_linCons_ep_prim} then becomes 

\be\label{discreteOTExtraLinConsEnt}
\begin{array}{lrll}
     \min & \sum_{x \in X} \ep c_x \ga_x &+ \eta \sum_{x \in X} \log\lt(\f{\ga_x}{\boldsymbol{\mu}_x}\rt) \ga_x  &  \\
     \text{subject to} & \sum_{x \in X|_{x_i}} \ga_x &= \mu_{x_i}^i & \forall x_i \in X_i, \quad i = 1,...,n\\
    & \sum_{x \in X} q_x^j \ga_x &= 0 & j = 1,...,K\\
    & \ga_x & \geq 0 & \forall x \in X,
\end{array}
\ee
while its corresponding dual problem becomes an unconstrained optimization problem:
\be\label{discreteOTExtraLinConsEntDual}
     \max_{\psi, p} \sum_i^n \sum_{x_i \in X_i} \psi_{x_i}^i \mu_{x_i}^i - \eta  \sum_{x \in X} \et \lt(\f{\sum_i^n \psi_{x_i}^i + \sum_{j=1}^K q_x^j p^j - \ep c_x}{\eta} \rt) \boldsymbol{\mu}_x.
\ee
Note that in this setting, \eqref{discreteOTExtraLinConsEnt} is known as an exponentially penalized linear program, and duality is well known (see for instance \cite{Cominetti1994}). It will be convenient to express the problem more compactly; we enumerate $X = \{x_\ell\}_{\ell=1}^m$, concatenate $\varphi:=(\psi,p)$ and denote product measure by $\boldsymbol{\mu} :=\ot_{i=1}^n \mu^i$, so that $\boldsymbol{\mu_\ell}= \ot_{i=1}^n \mu^i_{x_l^i}$ for all $x_l=(x^1_l,x^2_l,...,x^n_l) \in X$.  We then note that $ \sum_{i=1}^n \psi_{x^i_\ell}^i + \sum_{j=1}^Kq_{x_\ell}^j p^j=B_\ell\cdot \psi +C_\ell\cdot p=A_\ell \cdot \varphi$ for a vector $A_\ell=[B_\ell, C_\ell]$ and define the matrix $A=[B, C]$ whose rows are the $A_\ell$:
\be\label{eqn: contraint matrix A}
    A = \begin{bmatrix}
           A_1 \\
           \vdots \\
           A_{m}
         \end{bmatrix}.
\ee

It is easy to see that the above primal problem \eqref{discreteOTExtraLinConsEnt} 
can be expressed as:

\be\label{discreteOTExtraLinConsEntShort}
\begin{array}{lrll}
     \min & \multicolumn{2}{c}{\ep c^T \ga + \eta \sum_{\ell=1}^m \log(\f{\ga_\ell}{\boldsymbol{\mu}_\ell}) \ga_\ell} &  \\
     \text{subject to} & A^T \ga &= b\\
    & \ga& \geq 0 &, 
\end{array}
\ee
and the corresponding dual \eqref{discreteOTExtraLinConsEntDual} as:
\be\label{discreteOTExtraLinConsEntDualShort}
     \max_{\varphi}  b^T \varphi - \eta \sum_{\ell=1}^m \et\lt(\f{A_\ell \varphi - \ep c_\ell}{\eta} \rt) \boldsymbol{\mu}_\ell,
\ee
where $b=[\mu, 0]^T=[\mu^1\mu^2...\mu^n,0,...,0]^T$ is the vector formed by concatenating the vectors $\mu^i \in \mathbb{R}^{N^i}$ with a $0$ vector of dimension $K$ (the dimension of $Q$).

The dual program \eqref{discreteOTExtraLinConsEntDualShort} is not strictly concave in general.  A common strategy is to restrict the variable $\varphi$ to a subspace where strict concavity holds.  As we will see in the subsequent section, the following lemma implies that this can always be done without affecting the maximum in \eqref{discreteOTExtraLinConsEntDualShort} by fixing some components of $\varphi$ to be $0$ or, equivalently, removing some columns from $A$ and considering the reduced problem on the remaining corresponding components of $\varphi$. In the argument below, it will be useful to note that the matrix $B$ arises from marginal constraints and $C$ from additional linear constraints.  Therefore, for any $\gamma \in \Pi^Q(\mu^1,...,\mu^n)$ we have $\mu = \gamma^T B$ and $0 = \gamma^T C$.
\begin{lemma}\label{lem: reduction of A}
Assume that $\Pi^Q(\mu^1,\mu^2,...,\mu^n$) is non-empty. Then there is a matrix $\hat A=[\hat B, \hat C]$ and vector $\hat b =[\hat \mu, 0]^T$ of the corresponding dimensions such that the function

 \be\label{eqn: reduced optimization}
        \hat \varphi \mapsto \hat b^T \hat \varphi - \eta \sum_{\ell=1}^m \et\lt(\f{\hat A_\ell \hat \varphi - \ep c_\ell}{\eta} \rt) \boldsymbol{\mu}_\ell
    \ee
    has the same range as the objective function in \eqref{discreteOTExtraLinConsEntDualShort}, and $\hat A$ has full column rank.  Furthermore, $\hat B$ and $\hat C$ are obtained by removing some columns of $B$ and $C$ respectively and $\hat b$ by removing the corresponding entries of $b$.

\end{lemma}
\begin{proof}
It is easy to see that $\eqref{discreteOTExtraLinConsEntDualShort}$ and \eqref{eqn: reduced optimization} have the same range provided the matrices
    \be
        P = 
        \begin{bmatrix}
            b^T\\
            A\\
        \end{bmatrix}
        =
        \begin{bmatrix}
            \mathbf{\mu} & 0\\
            B & C\\
        \end{bmatrix}
        \text{\qquad and \qquad}
        \hat P = 
        \begin{bmatrix}
            \hat b^T\\
            \hat A\\
        \end{bmatrix} \nn
    \ee
have the same column space.

We first claim that if we drop columns of $B$ until the resulting matrix $\hat B$ has full column rank and the same range as $B$, and drop the corresponding entries of $ \mu$ to obtain $\hat \mu$, then $\begin{bmatrix}
           \hat{\mu} \\
            \hat B \\
        \end{bmatrix}$ has the same range as  $\begin{bmatrix}
            \mu \\
             B \\
        \end{bmatrix}$.  To see this, it suffices to show that if $\psi$ is in the null space of $B$, it is also orthogonal to $\mu$.  Noting that as $B$ captures the marginal constraints, we have $B_\ell \psi = \sum_{i=1}^n \psi_{x_\ell^i}$, where $x_\ell =(x_\ell^1,x_\ell^2,...,x_\ell^n)$, so if $B\psi=0$, $\sum_{i=1}^n \psi_{x_\ell^i}=0$ for each $l$.     
        
        Multiplying by 
        $\boldsymbol{\mu_\ell}$ and summing over $l$, we have $$
        0=\sum_{l=1}^m\boldsymbol{\mu_\ell}\sum_{i=1}^n \psi_{x_\ell^i}=\sum_{i=1}^n\sum_{l: x_\ell^i=x^i}\boldsymbol{\mu_\ell} \psi_{x^i}=\sum_{i=1}^n\mu^i_{x^i} \psi_{x^i}=\mu\psi,
        $$
        establishing the claim.

We now claim that we can drop a subset of the columns of $C$ to obtain a matrix $\hat C$ such that
$\begin{bmatrix}
           \hat{\mu} & 0\\
            \hat B & C \\
        \end{bmatrix}$ and
$\begin{bmatrix}
           \hat{\mu} & 0\\
            \hat B  & \hat C\\
            \end{bmatrix}$ have the same range while $[\hat B, \hat C]$ has full column rank.
Take $q^1$ as the first column of $C$. If $q^1$ does not belongs to the column space of $\hat B$, we set $\hat C^1 = q^1$. Otherwise, we claim that $\begin{bmatrix}
           \hat{\mu} \\
            \hat B \\
        \end{bmatrix}$ and
$\begin{bmatrix}
           \hat{\mu} & 0\\
            \hat B  & q^1\\
        \end{bmatrix}$ have the same range. If $\hat B \varphi = q^1$ and $\ga \in \Pi^Q(\mu^1,...,\mu^n)$, then $\hat{B}^T \ga = \hat{\mu}^T$ and ${q^1}^T \ga = 0$ by the marginal constraints and extra linear constraints, respectively. Hence we have 
  
        \be
            \hat{\mu} \varphi = \ga^T \hat B \varphi = \ga^T q^1 = 0.\nn
        \ee
        Then for any $\begin{bmatrix}
           u \\
            v \\
        \end{bmatrix}$ of the appropriate dimension, we have
        \be
        \begin{bmatrix}
           \hat{\mu} & 0  \\
            \hat B & q^1\\
        \end{bmatrix}
        \begin{bmatrix}
           u \\
            v \\
        \end{bmatrix} = 
        \begin{bmatrix}
           \hat{\mu} u  \\
            \hat B u + q^1 v\\  
        \end{bmatrix} = 
        \begin{bmatrix}
           \hat{\mu} (u + \varphi v) \\
            \hat B (u + \varphi v)\\  
        \end{bmatrix} =
        \begin{bmatrix}
           \hat{\mu}\\
            \hat B\\  
        \end{bmatrix}
        \begin{bmatrix}
           u + \varphi v \\ 
        \end{bmatrix},
        \ee
        where we have used the fact that $\hat{\mu} \varphi = 0$ and $\hat B \varphi = q^1$ in the second last equation. Therefore, we can conclude that it is safe to drop this column without changing the range. Continuing in this way, assume $\hat C^j$ is a matrix formed by the columns we keep after considering the first $j$ columns of $C$. If the $j+1$th column $q^{j+1}$ of $C$ does not belong to the column space of $[\hat B, \hat C^j]$, then we keep it and form a new matrix $\hat C^{j+1}$. Otherwise, we can write $\hat B \varphi + \hat C^j \psi = q^{j+1}$ for some $\varphi$ and $\psi$. Then, given that  $\hat{B}^T \ga = \hat{\mu}^T$ and $\hat{C}^{jT} \ga = 0$
        \be
            \hat{\mu} \varphi + 0 \psi = \ga^T \hat B \varphi + \ga^T \hat C^j \psi = \ga^T q^{j+1} = 0.
        \ee
        Then with the similar argument as above, $\begin{bmatrix}
           \hat{\mu} & 0\\
            \hat B & \hat C^j \\
        \end{bmatrix}$ and
$\begin{bmatrix}
           \hat{\mu} & 0 & 0\\
            \hat B  & \hat C^j & q^{j+1}\\
            \end{bmatrix}$ have the same range. Then we can set $\hat C^{j+1} = \hat C^j$ and in any case, $[\hat B, \hat C^{j+1}]$ has full column rank. We continuous this process until we check all columns of $C$ to obtain the final matrix $\hat C = \hat C^K$. By construction, $[\hat B, \hat C]$ has full column rank and 
            \be
    P = 
    \begin{bmatrix}
        \mu & 0\\
        B & C\\
    \end{bmatrix}
    \text{\qquad and \qquad}
    \hat P = 
    \begin{bmatrix}
        \hat \mu & 0\\
        \hat B & \hat C\\
    \end{bmatrix} \nn
\ee
have the same range. Now, we set 
    $\hat b^T = [\hat \mu,  0]$ and $\hat A = [\hat B, \hat C]$, and the result follows.
\end{proof}

\section{ODE for entropic Optimal Transport}\label{sect: ODE}

\subsection{Formulation of the ODE}\label{FormODESec}

We now formulate a system of ODEs characterizing the optimal dual Kantorovich potential for problem \eqref{discreteOTExtraLinConsEntDualShort}. For a fixed $\eta > 0$, we define the function:
\be\label{OTExLinConsFunc1}
    \Phi(\varphi, \ep) = -b^T \varphi + \eta \sum_{\ell=1}^m \et\lt(\f{A_\ell \varphi - \ep c_\ell}{\eta} \rt) \boldsymbol{\mu}_\ell.
\ee
By Lemma \ref{lem: reduction of A}, we can from now on always assume that $A$ has full column rank. We then consider the optimization problem:
\be\label{OTExLinConsProb}
    \min_\varphi \Phi(\varphi, \ep).
\ee
Then, \eqref{discreteOTExtraLinConsEntDualShort} is equivalent to \eqref{OTExLinConsProb} with the sign inverted. The first order optimality condition necessitates that at the minimizing $\varphi$:
\be\label{OTExLinConsFuncFirstOrd}
    \grad_\varphi \Phi(\varphi, \ep) = 0.
\ee

It is evident that $\Phi(\varphi, \ep)$ is smooth in $\varphi$ and $\ep$.   The lemma below implies its strict convexity in $\varphi$, and therefore, for each $\ep$, uniqueness of the solution $\varphi(\ep)$ to \eqref{OTExLinConsFuncFirstOrd}.  Invertibility of the Hessian (also implied by the following lemma) then ensures $\varphi(\ep)$ is smooth, by the implicit function theorem.
\begin{lemma}\label{ThmPhiConv}
    The function $\Phi(\varphi, \ep) = -b^T \varphi + \eta \sum_\ell^m \et (\foe (A_\ell \varphi - \ep c_\ell)) \boldsymbol{\mu}_\ell$ is uniformly convex in $\varphi$ over any compact domain $V \times T \subset \R^e \times \R$ where $e$ is the dimension of $\varphi$.
\end{lemma}
\begin{proof}
    A straightforward calculation shows that the gradient and Hessian of $\Phi(\varphi, \ep)$ in $\varphi$ are given by:
    \begin{align*}
        \grad_\varphi \Phi(\varphi, \ep) &= -b^T + \sum_{\ell=1}^m A_\ell^T  \et (\foe (A_\ell \varphi - \ep c_\ell)) \boldsymbol{\mu}_\ell\\
        D_{\varphi \varphi}^2 \Phi(\varphi, \ep) &= \foe \sum_{\ell=1}^m A_\ell^T \et (\foe (A_\ell \varphi - \ep c_\ell)) A_\ell \boldsymbol{\mu}_\ell. 
    \end{align*}
    The Hessian can be rewritten as:
    \be\label{discreteOTExtraLinConsEntHess}
        D_{\varphi \varphi}^2 \Phi(\varphi, \ep) = \foe A^T D A, 
    \ee
    where $D$ is a diagonal matrix with the $(\ell,\ell)$ th element being $\et (\foe (A_\ell \varphi - \ep c_\ell)) \boldsymbol{\mu}_\ell$. This implies that $D$ and therefore $A^T D A$ is positive definite, and as the entries are continuous, uniformly bounded below $A^T D A \geq d>0$ on the compact domain $V \times T$.
\end{proof}

Differentiating \eqref{OTExLinConsFuncFirstOrd} with respect to $\ep$ we obtain:
\be
    \f{\p}{\p \ep} \grad_\varphi \Phi(\varphi(\ep), \ep) + D_{\varphi \varphi}^2  \Phi(\varphi(\ep), \ep) \f{d \varphi}{d \ep}(\ep) = 0. \nn
\ee
Using the lemma, we obtain the Cauchy problem that governs the behavior of the optimal potential under the change of $\ep$:
\be\label{discreteOTExtraLinConsEntODE}
    \begin{cases}
        \f{d \varphi}{d \ep} (\ep) &= - \lt[ D_{\varphi \varphi}^2 \Phi(\varphi(\ep), \ep) \rt]^{-1} \f{\p}{\p \ep} \grad_\varphi \Phi(\varphi(\ep), \ep)\\
        \varphi(0) &= \phi_0.
    \end{cases}
\ee
Here, $\phi_0$ represents the initial condition of the ODE corresponding to the optimal $\varphi$ solving \eqref{OTExLinConsProb} when $\ep = 0$.  This corresponds to finding the measure $\gamma$ which minimizes the relative entropy among the admissable class $\Pi^Q(\mu^1,\mu^2,...,\mu^n)$, and is often easier to find than the solution to \eqref{OTExLinConsProb} for $\ep>0$. In particular, note that when  $Q = \{0\}$, representing the unconstrained (multi-marginal) optimal transport problem, the minimizer $\gamma_0$ is product measure, making it straightforward to verify that $\phi_0 = 0$.  
It is straightforward to calculate:
\be
    \f{\p}{\p \ep} \grad_\varphi \Phi(\varphi, \ep) = -\foe \sum_{\ell=1}^m c_\ell A_\ell^T  \et (\foe (A_\ell \varphi - \ep c_\ell)) \boldsymbol{\mu}_\ell.  \nn
\ee
Combined with equation \eqref{discreteOTExtraLinConsEntHess}
for the Hessian, the ODE \eqref{discreteOTExtraLinConsEntODE} can be written explicitly.

\subsection{Well-posedness of the ODE}
We now establish that the ODE \eqref{discreteOTExtraLinConsEntODE} has a unique solution.  This is a straightforward application of the Cauchy-Lipschitz Theorem, after the verification of a few preliminary properties.

As $\varphi(\ep)$ is clearly continuous,  there exist an $M > 0$ such that $||\varphi(\ep)|| \leq M, \ep \in [0, 1]$.  Therefore, we aim to prove the well-posedness of \eqref{OTExLinConsProb} over $V \times [0,1]$ where:
\be
    V \coloneqq \lt\{ \varphi \in \R^e| 
 \, ||\varphi|| \leq M \rt\}. \nn
\ee

\begin{theorem}\label{wellposeODE}
    Let $\varphi(\ep)$ be the solution of \eqref{OTExLinConsProb} for $\ep \in [0,1]$. Then the trajectory $\ep \mapsto \varphi(\ep)$ is smooth and is characterized as the unique solution to the initial value problem:
    \be
        \begin{cases}
            \f{d \varphi}{d \ep} (\ep) &= - \lt[ D_{\varphi \varphi}^2 \Phi(\varphi(\ep), \ep) \rt]^{-1} \f{\p}{\p \ep} \grad_\varphi \Phi(\varphi(\ep), \ep)\\
            \varphi(0) &= \phi_0
            \end{cases}.
    \ee
    Here, $\phi_0$ is the solution to the problem:
    \be
        \min_{\varphi} \Phi(\varphi, 0),
    \ee
    corresponding to the dual solution to the entropic minimization with respect to the product measure satisfying the extra linear constraints.
\end{theorem}
    
\begin{proof}
    The smoothness of $\varphi(\ep)$ and the fact that it satisfies the ODE is demonstrated in section \ref{FormODESec}. 

    For existence and uniqueness, to apply the Cauchy-Lipschitz Theorem, we only need that the function 
    \be
        (\varphi, \ep) \mapsto - \lt[ D_{\varphi \varphi}^2 \Phi(\varphi(\ep), \ep) \rt]^{-1} \f{\p}{\p \ep} \grad_\varphi \Phi(\varphi(\ep), \ep)\nn
    \ee
    is Lipschitz with respect to $\varphi$ and continuous with respect to $\ep$ on the domain $V \times [0,1]$.  This follows immediately from the smoothness and uniform convexity of $\Phi$ (noting that the latter property ensures an upper bound on  $[ D_{\varphi \varphi}^2 \Phi(\varphi(\ep), \ep) ]^{-1}$).
\end{proof}

\begin{remark}
   Due to the algebraic properties of the entropy, we can recover the curve of  the solution to the primal $\ga(\ep)$ by using the optimal dual variables, that is
    \be\label{primalRecover}
        \ga(\ep) = \exp \lt(\f{\sum_i^n \psi^i(\ep) + \sum_{j=1}^Kq^j p^j (\ep)- \ep c}{\eta} \rt) d(\ot_{i=1}^n \mu^i).\nn
    \ee
\end{remark}

The convergence of the Sinkhorn algorithm for unconstrained two or multi-marginal problems is well known \cite{carlier22}. On the other hand, while the Sinkhorn algorithm has been applied to various linearly constrained variants \cite{demarch2018}, \cite{eckstein2022}, there does not seem to be a general convergence result in this context.  The generalization of the Sinkhorn algorithm to this setting is a block coordinate descent method, and our Lemma's \ref{lem: reduction of A} and \ref{ThmPhiConv} imply that the entropic dual problem \eqref{discreteOTExtraLinConsEntDualShort} can always be formulated in a way that this scheme converges, as the following result confirms.

\begin{proposition}\label{lem: gen SK}
     For any initial $\phi_0$, the block coordinate descent iteration (Sinkhorn algorithm) will converge to the unique optimal solution of \eqref{OTExLinConsProb} with fixed $\ep \geq 0$.
\end{proposition}

\begin{proof}
    Recall that we are assuming that $A$ has full column rank (and that this assumption is justified by Lemma \ref{lem: reduction of A}).  By Lemma \ref{ThmPhiConv}, $\Phi(\varphi, \ep)$ is strongly convex in $\varphi$ for fixed $\ep \geq 0$. Together with the differentiability of $\Phi$ in $\varphi$, by Theorem 14.6.7 of \cite{OrtegaRheinboldt00}, the result follows.
\end{proof}
In particular, note that this result will be crucial in Section \ref{subsect: multi-period martingales} below, as, to the best of our knowledge, a version of the Sinkhorn algorithm has not been implemented on multi-period martingale optimal transport problems before.

\section{Derivatives of the optimal cost}
In this section, we illustrate how the ODE \eqref{discreteOTExtraLinConsEntODE} can be used to compute derivatives of the optimal cost at $\ep =0$.  For simplicity, we restrict our attention to the unconstrained two marginal optimal transport problem (that is, $n=2$, $Q=\{0\}$), although the techniques apply more generally.  In this setting, adopting slightly different notation than before, the primal problem is

\be \label{eqn: primal regularized 2 marginal OT}
    P(\ep) \coloneqq \inf_{\ga \in \Pi(\mu, \nu)} \int_{X \times Y} \ep c(x,y) d\ga(x,y) + \eta H_{\mu \ot \nu}(\ga),
\ee
while the dual is:
\be
    D(\ep) \coloneqq \sup_{u, v} \int_X u d\mu +  \int_Y v d\nu - \eta \int_{X \ot Y} e^{\f{u + v - \ep c}{\eta}} d\mu d\nu.
\ee
For the discrete marginals $\{\mu_r\}, \{\nu_s\}$, if we denote $c_{rs} = c(x_r, y_s)$, then the dual can be written as an equivalent problem:
\be
    C(\ep) \coloneqq \inf_{u, v} \Phi(u, v, \ep) = -D(\ep)=-P(\ep),
\ee
where
\be\label{OTdualcost1}
    \Phi(u, v, \ep) = -\sum_r u_r \mu_r - \sum_s v_s \nu_s + \eta \sum_{r,s} \et \lt[ \foe (u_r + v_s - \ep c_{rs}) \rt] \mu_r \nu_s.
\ee
 We will see that one can use the ODE approach to evaluate all derivatives of $C(\ep)$ at $\ep =0$ in closed form (we provide the explicit calculations for derivatives up to order $2$, but it will be clear that higher order derivatives can be calculated similarly).  A potential application is a Taylor expansion for $C(\ep)$ around $\ep =0$ to arbitrary order.

In this setting, it is actually simpler to eliminate $v$ and express the problem in terms of $u$ only.  Letting $(u(\ep),v(\ep))$ be the minimizer in \eqref{OTdualcost1}, or equivalently, solutions to our ODE, we have, at $(u,v) =(u(\ep),v(\ep))$, by the first order optimality condition, $\forall j$,
\be
    \f{\p}{\p v_j} \Phi = -\nu_j + \sum_r \et \lt[ \foe (u_r + v_j - \ep c_{rj})\rt] \mu_r \nu_j = 0. \nn
\ee
Hence we have
\be\label{OTVexpression}
    v_j = - \eta \log\lt[ \sum_r \et \lt[ \foe(u_r - \ep c_{rj}) \rt] \mu_r \rt].
\ee
After we substitute \eqref{OTVexpression} in \eqref{OTdualcost1}, we get:
\be
    \Phi(u, \ep) \coloneqq -\sum_r u_r \mu_r + \eta \sum_s \log\lt[ \sum_r \et \lt[ \foe(u_r - \ep c_{rs}) \rt] \mu_r \rt]\nu_s - \eta,
\ee
where we have abused the notation for $\Phi$; note that we can now write $C(\ep)=\inf_u \Phi(u,\ep)$.

 For simplicity, we denote $e_{ij} = \et[ \foe (u_i - \ep c_{ij})]\mu_i$. Note that,
\begin{align}
    -P'(\ep)=-D'(\ep) = C'(\ep) &= \f{d}{d \ep} \Phi(u(\ep), \ep) \nn\\
    &= \grad_u\Phi(u(\ep), \ep) \f{d u}{d \ep} + \f{\p}{\p \ep}\Phi(u(\ep), \ep) \nn\\
    &=\f{\p}{\p \ep}\Phi(u(\ep), \ep) \nn\\
    &= -\sum_s \lt[ \f{\sum_r e_{rs} c_{rs} }{\sum_r e_{rs}} \rt]\nu_s, \label{OT1storder}
\end{align}
where we have used the first order optimality condition $\grad_u\Phi(u(\ep), \ep) = 0, \forall \ep \geq 0$. Since it is easy to verify that $\forall i, u_i(0) = 0$, we have $e_{ij} = \mu_i$ when $\ep = 0$, therefore: 
\be\label{OT1storder0}
    C'(0) = -\sum_s \lt[ \f{\sum_r c_{rs} \mu_r}{\sum_r \mu_r} \rt] \nu_s = -\E[c(X, Y)].
\ee
The expectation in \eqref{OT1storder0} is taken with respect to the product measure $\mu \ot \nu$. \footnote{Note that \eqref{OT1storder0} could more easily be seen by applying the envelope theorem to the primal problem \eqref{eqn: primal regularized 2 marginal OT}  and noting the unique optimality of product measure  at $\ep=0$ in that equation.  However, we will require \eqref{OT1storder} to compute higher order derivatives.}

For the second order derivative,
\begin{align}
    C''(\ep) &= \f{d}{d \ep} \lt( \f{\p}{\p \ep}\Phi(u(\ep), \ep) \rt) \nn\\
    &= \lt[ \grad_u\f{\p}{\p \ep}\Phi(u(\ep), \ep)\rt]^T \f{d u}{d \ep} + \f{\p^2}{\p \ep^2}\Phi(u(\ep), \ep) \nn\\
    &= \lt[ \grad_u\f{\p}{\p \ep}\Phi(u(\ep), \ep)\rt]^T \lt[ -D^2_{uu} \Phi(u(\ep), \ep) \rt]^{-1}\lt[ \grad_u\f{\p}{\p \ep}\Phi(u(\ep), \ep)\rt] + \f{\p^2}{\p \ep^2}\Phi(u(\ep), \ep).
     \label{OT2ndorder}
\end{align}
The last equality is due to the ODE  \eqref{discreteOTExtraLinConsEntODE}. We note that evaluating this quantity at the point $\ep=0$ (where $u(0) =0$ is known) involves only derivatives of $\Phi$.  We relegate the detailed calculation to Appendix \ref{AppC2ndDerCal} and provide only the result here:

\be \label{eqn: second derivative of optimal cost at 0}
    C''(0) = \foe \biggl( \lt(\E[c(X, Y)] \rt)^2 + \E[c^2(X, Y)] - \E \lt[(\E [c(X, Y)| X])^2\rt] - \E \lt[(\E [c(X, Y)| Y])^2\rt]\biggr), 
\ee
where the expectations are with respect to product measure $\mu \otimes \nu$, or, equivalently,

\begin{align*}
   C''(0) =  &\frac{1}{\eta} \Biggl( \lt(\int_{X \times Y} c(x,y) d\mu(x)d\nu(y) \rt)^2 + \int_{X \times Y} c^2(x,y) d\mu(x)d\nu(y) \\
    - & \int_Y \lt(\int_X c(x,y) d\mu(x) \rt)^2 d\nu(y) - \int_X \lt(\int_Y c(x,y) d\nu(y) \rt)^2 d\mu(x) \Biggr).
\end{align*}
We note that higher order derivatives can be evaluated similarly; for instance, one can find $C'''(\ep)$ by differentiating \eqref{OT2ndorder} using the chain rule, and then use the ODE \eqref{discreteOTExtraLinConsEntODE} to eliminate the $\frac{du}{d\ep}$ terms and substitute $u(0)=0$ to determine $C'''(0)$.  The result is a closed form expression involving only derivatives of the function $\Phi$ at $(u,\ep)=(0,0)$, which can readily be computed.

\begin{remark}
 $P(\ep)$ as defined by \eqref{eqn: primal regularized 2 marginal OT} is clearly concave, as an infimum of affine functions.  Therefore, we must have $C''(0) =-P''(0) \geq 0$.

 We do not see another way to prove that the expression in \eqref{eqn: second derivative of optimal cost at 0} is non-negative.  However, the fact that expectations are taken with respect to product measure is crucial, as the expression can in fact be negative for more general couplings.  To see this, note that 
    \begin{align*}
        \text{Var}[\E[c(X,Y)|X]] &= \E[(\E[c(X, Y)| X])^2] - (E[c])^2\\
        \E[\text{Var}[c(X,Y)|X]] &= \E[c^2] - \E[(\E[c(X, Y)| X])^2] ,
    \end{align*}
    and consider the case $\mu = \nu$ such that $\mu$ is supported on more than one point, $c(x, y) = x + y$ and the coupling is given by $X=Y$. In this case, $\E[(\E[c(X, Y)| Y])^2] = \E[(\E[c(X, Y)| X])^2]$ and the non-negativity of the quantity in \eqref{eqn: second derivative of optimal cost at 0} becomes
    \be
        \E[\text{Var}[c(X,Y)|X]] \geq \text{Var}[\E[c(X,Y)|X]]. \nn
    \ee
    However, it is easy to calculate that $\E[\text{Var} [c(X,Y)|X]] = 0$ but $\text{Var}[\E[c(X,Y)|X]] = 4 \text{Var}[X] \neq 0$, which violates the above inequality.
\end{remark}

\section{Numerical examples}\label{numericSec}
In this section we exploit the ODE \eqref{discreteOTExtraLinConsEntODE} to compute numerical approximations of solutions to the optimal transport problem and several variants.  In all the examples below, we discretize the ODE in $\ep$ and solve using a 4th-order Runge-Kutta method.

To verify the accuracy of our new numerical method, we also compute solutions via the Sinkhorn algorithm in each example in the first four subsections below; in all cases, the values of the optimal costs obtained from the Sinkhorn are very close to the values obtained from the ODE method\footnote{ A similar comparison is not provided for the example in the last subsection, since its purpose is somewhat different; in particular, one of the marginals of the curve $\gamma(\ep)$ is in fact the main object of interest there, rather than the final value $\gamma(1)$, and this curve of marginals is not as readily obtained by the Sinkhorn method.}.  Furthermore, we have chosen examples for which explicit solutions $\gamma_0$ of the unregularized problem \eqref{eqn: unregularized primal} are known. This then yields a lower bound on the optimal value in \eqref{EOT_linCons_prim}.  Since $\gamma_0$ is a competitor in \eqref{EOT_linCons_prim}, it can also be used to compute an upper bound for this problem; thus, we obtain an interval in which the optimal value must fall:
\begin{equation}\label{general optimal range}
 \int_X c d\ga_0\leq  \inf_{\ga \in \Pi^Q(\mu^1,...,\mu^n)} \int_X c d\ga + \eta H_{\ot_{i=1}^n \mu^i}(\ga) \leq   \int_X c d\ga_0 + \eta H_{\ot_{i=1}^n \mu^i}(\ga_0).
\end{equation}
In the examples in the first four subsections below, the values obtained via both the Sinkhorn and ODE method do indeed fall with this interval.  

Throughout this section, we will consider problems with at most three marginals which will be represented by the symbols $\mu$, $\nu$ and $\theta$.  We will denote the corresponding spaces $X$, $Y$ and $Z$.  
We perform all numerical calculations in Python on 13th Gen Intel(R) Core(TM) i7-13620H 2.40 GHz Notebook.

\subsection{Two marginal optimal transport}

We begin with the two marginal optimal transport problem.  Let $\mu = \sum_r \mu_r \de_{x_r}, \nu = \sum_s \nu_s \de_{y_s}$ be discrete measures and let the cost matrix be $c_{rs} = c(x_r, y_s)$. For fixed $\eta > 0$, we aim to solve:
\be 
    \inf_{u,v} \Phi(u, v, \ep),
\ee
where
\be\label{OTdualNumeric}
    \Phi(u, v, \ep) = - \sum_r u_r \mu_r - \sum_s v_s \nu_s + \eta \sum_{r,s} \et \lt[\foe (u_r + v_s - \ep c_{rs})\rt] \mu_r \nu_s.
\ee
It well known that if we fix $u_0 = v_0 = 0$, then there will be a unique minimizer of \eqref{OTdualNumeric}. This corresponds to removing the corresponding column to the $A$ in \eqref{OTExLinConsFunc1} so that the reduced version has full column rank, as discussed more generally in Lemma \ref{lem: reduction of A}.
 
As in the previous section, it will be convenient to further reduce the number of variables, using the first order optimality condition:
\begin{align}
    v_j &= - \eta \log\lt[ \sum_r \et \lt[ \foe(u_r - \ep c_{rj}) \rt] \mu_r \rt] =:v_j(u).\label{OTdual1storderV}
\end{align}

We simplify by substituting \eqref{OTdual1storderV} into \eqref{OTdualNumeric} to yield the new dual objective function: 
\be\label{OTmodifiedDual}
    \bar \Phi(u, \ep) = \Phi(u, v(u), \ep) = - \sum_r u_r \mu_r + \eta \sum_s \log \lt[\sum_r \et \lt[\foe (u_r - \ep c_{rs}) \rt] \mu_r \rt] \nu_s + \eta.
\ee
The unique minimizer $u(\ep)$ of $u \mapsto \bar \Phi(u, \ep)$ will clearly still satisfify the ODE  \eqref{discreteOTExtraLinConsEntODE} (with $\bar \Phi$ in place of $\Phi$).

\begin{remark}\label{rem: reduction in number of variables}
The ODE induced by the function \eqref{OTmodifiedDual}  inherits the well-posedness related to the original function \eqref{OTdualNumeric}. The gradient and the Hessian of the function \eqref{OTmodifiedDual} are:
\begin{align*}
    \grad \bar \Phi(u, \ep) &= \grad_u \Phi(u, v(u), \ep) + \grad_v \Phi(u, v(u), \ep) D v(u)\\
    D^2 \bar \Phi(u, \ep) &= D_{uu}^2  \Phi(u, v(u), \ep) + D_{uv}^2  \Phi(u, v(u), \ep) D v(u)+ D v(u)^T D_{vu}^2  \Phi(u, v(u), \ep)\\
    & \qquad + D v(u)^T D_{vv}^2  \Phi(u, v(u), \ep) D v(u) + \grad_v \Phi(u, v(u), \ep) D^2 v(u).
\end{align*}
The last term vanishes since $\grad_v \Phi = 0$. Therefore, we have:
\be
    D^2 \bar \Phi = [I, D v^T] D^2 \Phi \begin{bmatrix}
    I \\
    D v \\
    \end{bmatrix}.
\ee
$D^2\bar \Phi$ then inherits lower bounds from $D^2 \Phi$, and so $\bar \Phi$ is uniformly convex on compact sets.  The well posedness of the ODE satisfied by $u$ then follows exactly as in Theorem  \ref{wellposeODE}.  We will abuse notation below by referring to the function of $u$ and $\ep$ only as $\Phi$ (rather than $\bar \Phi$ as above), and, in subsequent subsections, we will make similar reductions in the number of variables; the well-posedness of the corresponding ODEs will follow similarly and we will not comment on it in detail.
\end{remark} 

Let 
\be
    e_{ij} = \et \lt[\foe (u_i - \ep c_{ij} ) \rt] \mu_i, \qquad e_{ij}^c = \et \lt[\foe (u_i - \ep c_{ij} ) \rt] c_{ij} \mu_i. \nn
\ee
Then the Hessian matrix for $\Phi$ will be the symmetric matrix:
\be
    \biggl(D_{uu}^2 \Phi(u(\ep), \ep) \biggr)_{i_1i_2} = \foe \sum_s \lt[\f{e_{i_1s}}{\sum_r e_{rs}} \de_{i_1 i_2} - \f{e_{i_1s} \cdot e_{i_2s}}{(\sum_r e_{rs})^2} \rt] \nu_s,
\ee
where $\de_{i_1i_2}$ is the Kronecker delta and the mixed derivative is:
\be
    \biggl( \grad_u \f{\p}{\p \ep} \Phi(u(\ep), \ep) \biggr)_i = -\foe \sum_s \lt[\f{e_{is}^c}{\sum_r e_{rs}} - \f{e_{is} \cdot e_s^c}{(\sum_r e_{rs})^2} \rt] \nu_s.
\ee
It is easy to verify that the initial condition for the system of equations is $u_i = 0, \forall i$. 

We take $\eta = 0.002$ and discretize $[0,1]$ with 100 time steps for the variable $\ep$. Both the marginals $\mu$ and $\nu$ are uniformly supported on 100 evenly spaced grid points on $[0,1]$. We solved the ODE with two costs: the attractive cost $c(x,y) = (y - x)^2$ and the repulsive cost $c(x,y) = -\log(0.1 + |x - y|)$. It is well known that the optimal measure for the attractive cost is induced by the identity map in our setting. A closed form for the optimal measure for the repulsive cost is also known \cite{colomboDePascaleDiMarino2015}. We provide the evolution of the primal coupling at 16 evenly distributed moments for $\ep$ between 0 and 1 for both attractive and repulsive type costs in Figure \ref{OTAttractgaEvolve}, and \ref{OTRepulgaEvolve4}, respectively.  Note that the solutions at $\ep =1$ agree qualitatively with the known solutions for both costs. Tables \ref{OTDataTable4} and \ref{OTDataTable8} summarize the corresponding numerical calculations for the  attractive and repulsive type costs, respectively. In both cases, the optimal values are similar to those found by the Sinkhorn algorithm, and lie within the range given by \eqref{general optimal range}.

\begin{figure}
    \centering
    \includegraphics[width=1\linewidth]{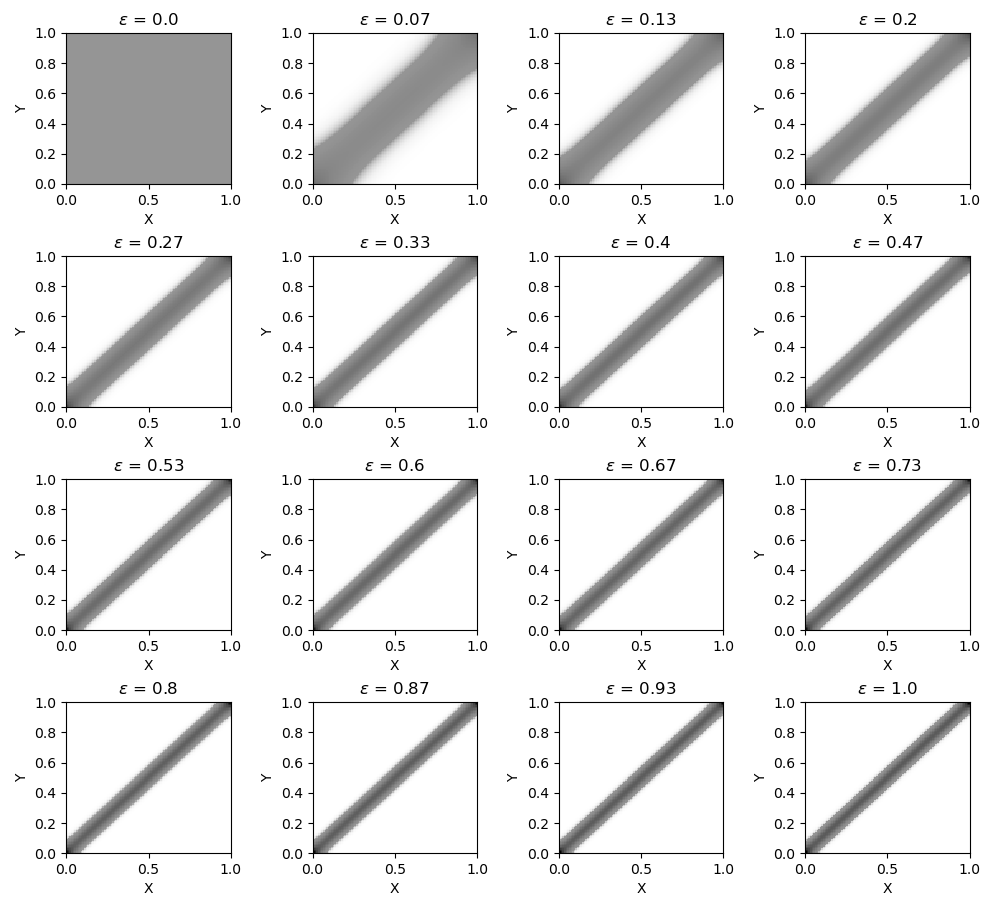}
    \caption{Evolution of the primal coupling at 16 evenly distributed values of $\ep$ between 0 and 1 calculated by the ODE method for two marginal optimal transport with attractive cost $c(x, y) = |y - x|^2$.}
    \label{OTAttractgaEvolve}
\end{figure}

\begin{figure}[h!]
    \centering
    \includegraphics[width=1\linewidth]{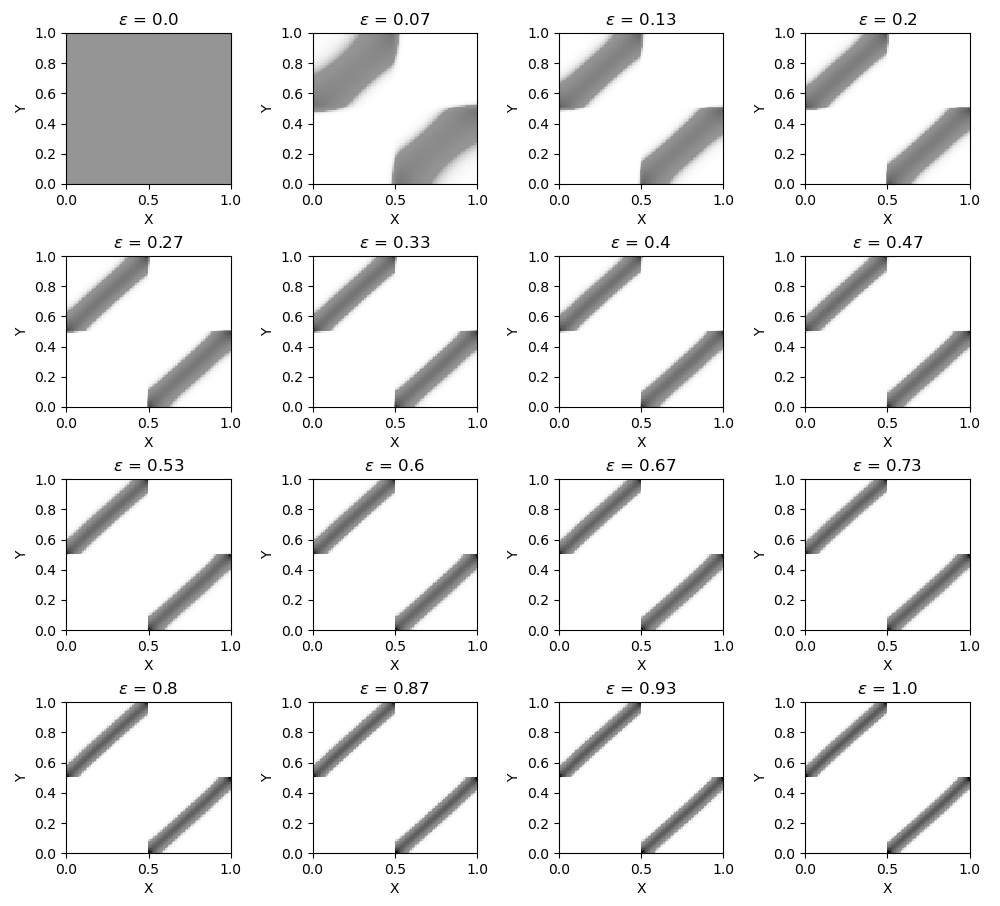}
    \caption{Evolution of the primal coupling at 16 evenly distributed values of $\ep$ between 0 and 1 calculated by the ODE method for two marginal optimal transport with repulsive cost $c(x, y)  = -\log(0.1 + |x - y|)$.}
    \label{OTRepulgaEvolve4}
\end{figure}

\begin{table}
    \centering
    \begin{tabular}{|c|c|c|}
        \hline
                       & ODE & Sinkhorn \\ \hline
        Computed regularized primal value  & 0.0050  & 0.0052\\ \hline
        Optimal unregularized primal value   & 0  & 0\\ \hline
        Optimal unregularized primal value + Entropy  & 0.0092  & 0.0092 \\ \hline
        Iterations     & 100      & 324   \\ \hline
        CPU time (sec) & 1.40    & 0.04   \\ \hline
    \end{tabular}
    \caption{Comparison of the performance between 4-th order Runge-Kutta ODE method and Sinkhorn algorithm for two marginal optimal transport with attractive cost $c(x, y)  = (y - x)^2$.}
    \label{OTDataTable4}
\end{table}

\begin{table}
    \centering
    \begin{tabular}{|c|c|c|}
        \hline
                       & ODE & Sinkhorn \\ \hline
        Computed regularized primal value  & 0.5033  & 0.5080\\ \hline
        Optimal unregularized primal value   & 0.5024  & 0.5024\\ \hline
        Optimal unregularized primal value + Entropy  & 0.5117  & 0.5117 \\ \hline
        Iterations     & 100      & 684   \\ \hline
        CPU time (sec) & 1.49    & 0.15   \\ \hline
    \end{tabular}
    \caption{Comparison of the performance between 4-th order Runge-Kutta ODE method and Sinkhorn algorithm for two marginal optimal transport with repulsive cost $c(x, y)  = -\log(0.1 + |x - y|)$.}
    \label{OTDataTable8}
\end{table}

\subsection{Multi-marginal optimal transport}
We turn now to a three marginal optimal transport problem; in this setting our approach here is closely related to the work in \cite{nenna2022}.  However, the present framework can accommodate  completely general cost functions (whereas the work in \cite{nenna2022} applies only to pairwise costs). Even for pairwise costs, the solution $\gamma(\ep)$ as a function of the regularization parameter here is likely of more direct interest than the interpolation between two and multi-marginal solutions produced by the method in \cite{nenna2022}.

Let $\mu = \sum_r \mu_r \de_{x_r}, \tht = \sum_s \tht_s \de_{y_s}, \nu = \sum_t \nu_t \de_{z_t}$ and $c_{rst} = c(x_r, y_s, z_t)$ be the cost tensor. For fixed $\eta > 0$, we minimize the following objective function:
\be\label{MMargOTdualNumeric}
    \Phi(u, v, w, \ep) = -\sum_r u_r \mu_r - \sum_s v_s \tht_s - \sum_t w_t \nu_t + \eta \sum_{r,s,t} \et \lt[\foe(u_r + v_s + w_t - \ep c_{rst}) \rt] \mu_r \tht_s \nu_t.
\ee
As in the previous section, we substitute the first order condition with respect to one variable,
\begin{align}
    w_k &= - \eta \log\lt[\sum_{rs} \et\lt[\foe (u_r + v_s - \ep c_{rsk}) \rt]\mu_r \tht_s \rt] ,\label{MMargOTdualW}
\end{align}
into \eqref{MMargOTdualNumeric} to reduce the number of unknown variables:
\be
    \Phi(u, v, \ep) = -\sum_r u_r \mu_r - \sum_s v_s \tht_s + \eta \sum_t \log \lt[\sum_{r,s} \et \lt[\foe (u_r + v_s - \ep c_{rst}) \rt] \mu_r \tht_s \rt] \nu_t + \eta.
\ee
As in the two marginal case (see Remark \ref{rem: reduction in number of variables}) , the minimizers $u(\ep, v(\ep))$ will satisfy the well-posed ODE \eqref{discreteOTExtraLinConsEntODE}.

We refer the reader to appendix \ref{AppMMargOT2ndDerSec} for the Hessian of $\Phi$ with respect to the dual potential variables $\varphi=(u, v)$ and the mixed derivative with respect to $(u, v)$ and $\ep$. 

We set the initial values  to be $u_i = v_j = 0, \forall i, j$ which can be verified to be the unique solution for $\ep = 0$ satisfying  our normalization $u_0 = v_0 = 0$. For the numerical simulation, each marginal to be uniform on 99 evenly spaced grid points (for simplicity of calculating the true optimal measure \cite{colomboDePascaleDiMarino2015}) on $[0,1]$, $\eta = 0.006$ and $\ep \in [0,1]$ with 100 time steps. The cost function is $c(x, y, z) = d(x,y) + d(y, z) + d(x, z)$ where $d(x,y) = -\log(0.1 + |x - y|)$, which is the same as the first example in \cite{nenna2022}. Figure \ref{ThreeMargOTga} shows the 2D projection to all the pairs of marginals at $\ep = 1$, which agrees well with the known explicit solution to the unregularized problem in \cite{colomboDePascaleDiMarino2015}. From Table \ref{threemarginalDataTable}, the values obtained from the ODE and Sinkhorn method agree quite well and both fall within the interval given by \eqref{general optimal range}.  Note in addition that in this case, the ODE method actually takes less running time and fewer iterations than the Sinkhorn.

\begin{figure}[h!]
    \centering
    \includegraphics[width=1\linewidth]{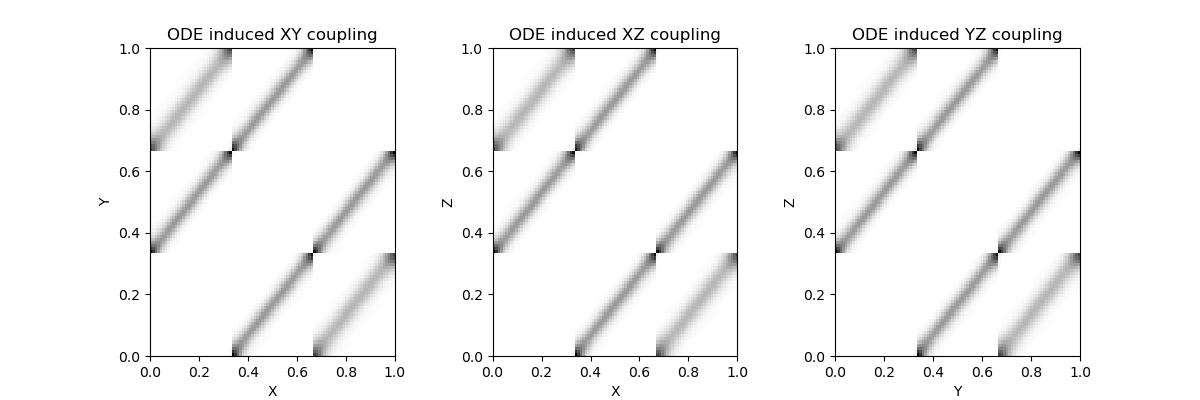}
    \caption{The optimal measure for three marginal optimal transport with cost $c(x, y, z) = d(x,y) + d(y, z) + d(x, z)$ where $d(x,y) = -\log(0.1 + |x - y|)$ generated by ODE}
    \label{ThreeMargOTga}
\end{figure}

\begin{table}
    \centering
    \begin{tabular}{|c|c|c|c|c|}
        \hline
                       & ODE & Sinkhorn \\ \hline
        Computed regularized primal value  & 1.9163  & 1.9193\\ \hline
        Optimal unregularized primal value   & 1.9137  & 1.9137\\ \hline
        Optimal unregularized primal value + Entropy  & 1.9647  & 1.9647 \\ \hline
        Iterations     & 100      & 1006   \\ \hline
        CPU time (sec) & 12.40    & 47.46   \\ \hline
    \end{tabular}
    \caption{Comparison of the performance between 4-th order Runge-Kutta ODE method and Sinkhorn algorithm for three marginal optimal transport problem.}
    \label{threemarginalDataTable}
\end{table}

\subsection{Martingale optimal transport}
Given probability  measures $\mu$ and $\nu$, with $\mu$ dominated by $\nu$ in convex order (denoted by $\mu \preceq_c \nu)$, the (one-period) martingale optimal transport falls into the class of optimal transport problems with extra linear constraints \eqref{eqn: unregularized primal} when we take the subspace $Q$ to be the set $M$ of martingale test functions:
\be
    M = \{g(x)(y - x)| g(x) \in  C_b(X)\}.
\ee
Note that $\int m d\pi = 0, \forall m \in M$ is equivalent to the classical martingale condition:
\be
    \int_Y y d\pi_x(y) = x,
\ee
where $\pi = \pi_x \ot \mu$ is the regular disintegration of $\pi$ with respect to $\mu$, or in probabilistic notation:
\be
    \E_\pi[Y|X] = X.
\ee
The primal problem of martingale optimal transport problem is:
\be
    \inf_{\pi \in \Pi^M(\mu, \nu)} \int_{X \times Y} c(x, y) d\pi(x, y),
\ee
and the corresponding dual problem:
\be
    \sup \int_X u(x) d\mu(x) + \int_Y v(y) d\nu(y), 
\ee
subject to $u(x) + v(y) + g(x)(y - x) \leq c(x, y)$, among all $u \in L(\mu), v \in L(\nu), g \in C(X)$. For the numerical example, let $\mu = \sum_r \mu_r \de_{x_r}, \nu = \sum_s \nu_s \de_{y_s}$ be discrete measures such that $\mu$ is less than $\nu$ in convex order. Then the cost matrix will be $c_{rs} = c(x_r, y_s)$. The entropic regularized problem with fixed $\eta$ will be:
\be
    \inf_{u,v,g} \Phi(u, v, g, \ep),
\ee
where
\be\label{MOTdualcost1}
    \Phi(u, v, \ep) = -\sum_r u_r \mu_r - \sum_s v_s \nu_s + \eta \sum_{r,s} \et \lt[ \foe (u_r + v_s + g_r(y_s - x_r)- \ep c_{rs}) \rt] \mu_r \nu_s.
\ee
The first order optimality condition implies
\begin{align}
    u_i &= - \eta \log\lt[ \sum_s \et \lt[ \foe(v_s + g_i(y_s - x_i) - \ep c_{is}) \rt] \nu_s \rt]\label{MOTdual1storderU}\\
    v_j &= - \eta \log\lt[ \sum_r \et \lt[ \foe(u_r + g_r(y_j - x_r) - \ep c_{rj}) \rt] \mu_r \rt]\\
    0 &= \sum_s \et \lt[ \foe(u_i + v_s + g_i (y_s - x_i) - \ep c_{is}) \rt] (y_s - x_i) \nu_s \label{MOTdual1storderG}.
\end{align}
We will use this set of equations in the Sinkhorn algorithm. One important note should be made towards the fitting of $g_i$ from $u_i$ and $v_i$ is not as straight forward as in the unconstrained optimal transport. Instead, we need to solve the root of the equation \eqref{MOTdual1storderG} for each $i$. Slightly longer computation time is expected if we apply Newton's method for finding the root.

On the other hand, we can further simplify \eqref{MOTdualcost1} by substituting \eqref{MOTdual1storderU}  into it. By abuse of notation, the dual cost becomes
\be\label{MOTdualcost2}
    \Phi(v, g, \ep) \coloneqq -\sum_s v_s \nu_s + \eta \sum_r \log\lt[ \sum_s \et \lt[ \foe(v_s + g_r(y_s - x_r) - \ep c_{rs}) \rt] \nu_s \rt]\mu_r + \eta.
\ee
The minimizers $v(\ep), g(\ep)$ satisfy the well-posed ODE \eqref{discreteOTExtraLinConsEntODE}.

We refer the reader to Appendix \ref{AppMOT2ndDerSec} for the Hessian of $\Phi$ with respect to the dual potential variables $v, g$ and the mixed derivatives with respect to $v, g$ and $\ep$. Unlike in the unconstrained optimal transport problem, $v_j = g_i = 0$ is no longer an initial condition for the system of equations since the product measure is in general not a martingale measure (nevertheless, setting $v_0 = g_0 = 0$ will reduce $A$ from \eqref{eqn: contraint matrix A} into a full column rank matrix). Instead, we can solve the initial system of equations, \eqref{MOTdual1storderU} - \eqref{MOTdual1storderG} with $\ep = 0$, by the Sinkhorn algorithm.  Note that by our numerical experiment, the running time for the Sinkhorn algorithm is much faster for $\ep = 0$ then $\ep = 1$. This keeps the overall calculation time for the ODE method competitive with the pure Sinkhorn method to minimize the $\ep=1$ objective function.

The cost function is $c(x, y) = \et(-x) \cdot y^2$ which is of martingale Spence–Mirrlees type, hence  the unregularized optimal coupling is the left-monotone coupling \cite{beiglbock2016}. The $x$-marginal $\mu$ is uniformly distributed on 100 evenly spaced grid points on $[-0.3, 0.3]$ while the $y$-marginal $\nu$ is uniformly distributed on 200 evenly spaced grid points on $[-1, 1]$. We set $\eta = 0.006$ and $\ep = 1$ for this numerical example. We employ 25 steps to discretize $\ep$ on $[0,1]$. The plot in Figure \ref{onePerMartPlotga}, the plot shows that the ODE method does indeed approximate the left-monotone measure. From Table \ref{threemarginalDataTable}, we see that both the ODE method and Sinkhorn algorithm yield comparable optimal values, and both fall inside the desired interval \eqref{general optimal range}.  Again, in this case, the ODE method takes fewer iterations and less running time.

\begin{figure}[h!]
    \centering
    \includegraphics[width=0.5\linewidth]{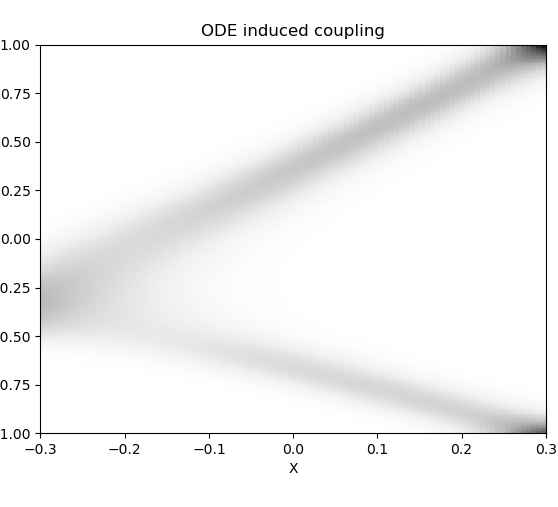}
    \caption{The optimal measure for martingale optimal transport with cost $c(x, y) = \et(-x) \cdot y^2$ generated by ODE}
    \label{onePerMartPlotga}
\end{figure}

\begin{table}
    \centering
    \begin{tabular}{|c|c|c|c|c|}
        \hline
                       & ODE & Sinkhorn \\ \hline
        Computed regularized primal value  & 0.2990  & 0.2990\\ \hline
        Optimal unregularized primal value   & 0.2964  & 0.2964\\ \hline
        Optimal unregularized primal value + Entropy  & 0.3211  & 0.3211 \\ \hline
        Iterations     & 25      & 152   \\ \hline
        CPU time (sec) & 1.44    & 3.38   \\ \hline
    \end{tabular}
    \caption{Comparison of the performance between 4-th order Runge-Kutta ODE method and Sinkhorn algorithm for martingale optimal transport problem.}
    \label{threemarginalDataTable}
\end{table}

\subsection{Multi-period martingale optimal transport}\label{subsect: multi-period martingales}
We consider a 3-period martingale optimal transport problem.  Let $\mu = \sum_r \mu_r \de_{x_r} \in \mathcal{P}(X), \tht = \sum_s \tht_s \de_{y_s} \in \mathcal{P}(Y), \nu = \sum_t \nu_t \de_{z_t} \in \mathcal{P}(Z)$ such that $\mu \preceq_c \tht \preceq_c \nu $. The linear constraints are:
\be
   Q= MM = \{g(x)(y - x) + h(x,y)(z - y)| g(x) \in C_b(X), h(x,y) \in C_b(X,Y)\}. \nn
\ee
Note that 
\be
    \E_\pi[Y|X] = X, \qquad \E_\pi[Z|X,Y] = Y \qquad \iff \pi \in \Pi^{MM}(\mu, \tht, \nu). \nn
\ee

With the cost tensor $c_{ijk} = c(x_i, y_j, z_k)$, the entropic regularized problem with fixed $\eta$ for the dual problem is
\be
   \inf_{u,v,w,g,h} \Phi(u, v, w, g, h, \ep) \nn
\ee
where
\begin{align}\label{MMOTdualCost1}
    \Phi(u, v, w, g, h, \ep) = &- \sum_r u_r \mu_r - \sum_s v_s \tht_s - \sum_t w_t \nu_t \\
    &\qquad + \eta \sum_{r,s,t} \et \biggl[ \foe \bigl( u_r + v_s + w_t + g_r(y_s - x_r) \nn\\ 
    &\qquad \qquad + h_{rs}(z_t - y_s) - \ep c_{rst} \bigr) \biggr] \mu_r \tht_s \nu_t. \nn
\end{align}

To implement the Sinkhorn algorithm, we need the first order optimality condition:
\begin{align}
    u_i &= - \eta \log\lt[\sum_{st} \et\lt[\foe (v_s + w_t + g_i(y_s - x_i) + h_{is}(z_t - y_s) - \ep c_{ist}) \rt]\tht_s \nu_t \rt] \label{MMOTdualU}\\
    v_j &= - \eta \log\lt[\sum_{rt} \et\lt[\foe (u_r + w_t + g_r(y_j - x_r) + h_{rj}(z_t - y_j)) - \ep c_{rjt} \rt]\mu_r \nu_t \rt] \nn \\
    w_k &= - \eta \log\lt[\sum_{rs} \et\lt[\foe (u_r + v_s + g_r(y_s - x_r) + h_{rs}(z_k - y_s) - \ep c_{rsk}) \rt]\mu_r \tht_s \rt] \nn 
\end{align}
\begin{align}
    0 &= \sum_{s,t}  \et\lt[\foe (u_i + v_s + w_t + g_i(y_s - x_i) + h_{is}(z_t - y_s) - \ep c_{ist}) \rt] (y_s - x_i) \tht_s \nu_t \nn \\
    0 &= \sum_{t} \et\lt[\foe (u_i + v_j + w_t + g_i(y_j - x_i) + h_{ij}(z_t - y_j)) - \ep c_{ijt} \rt] (z_t - y_j) \nu_t. \nn
\end{align}
The fitting of $u_i, v_j, w_k$ are straight forward while we use Newton's method to solve for $g_i$ and $h_{ij}$ in each of the iteration.

On the other hand, by combining \eqref{MMOTdualCost1}
and \eqref{MMOTdualU}, with a slight abuse of notation again we set:
\begin{align}\label{MMOTdualCost2}
    \Phi(v, w, g, h, \ep) = &- \sum_s v_s \tht_s - \sum_t w_t \nu_t + \eta \sum_r \log \biggl[ \sum_{s,t} \et \biggl[ \foe \bigl( v_s + w_t\\
    &\qquad + g_r(y_s - x_r) + h_{rs}(z_t - y_s) - \ep c_{rst} \bigr) \biggr] \tht_s \nu_t \biggr] \mu_r + \eta, \nn
\end{align}
and note that the minimizing $v(\ep),w(\ep),g(\ep),h(\ep)$ solve the well-posed ODE \eqref{discreteOTExtraLinConsEntODE} after we fix $v_0, w_0, g_0$ and $h_{00}$ all to be 0.
We refer the reader to Appendix \ref{AppMMOT2ndDerSec} for the Hessian of $\Phi$ with respect to the dual potential variables $v, w, g, h$ and the mixed derivatives with respect to $v, w, g, h$ and $\ep$. 

For the numerical experiment, we took the marginals to be uniform on 30, 60 and 90 evenly spaced grid points on $X=[-0.1,0.1]$, $Y =[-0.4, 0.4]$ and $Z=[-1,1]$, respectively. 
We set $\eta = 0.006$ and $\ep \in [0,1]$ with 25 time steps. The cost function is $c(x, y, z) = y^2\et(-x) + z^2\et(-x)$ which is of martingale Spence-Mirrlees type. It is known that the optimal measure is the multi-period left-monotone coupling \cite{NutzStebeggTan20}. Since $\mu$ is a discrete measure, the multi-period left-monotone coupling is non-unique, but we know from  Lemma \ref{lem: convg min ent sol}, the optimal measures for regularized problems will converge to the left-monotone coupling with minimal entropy.  Explicitly, the left-monotone condition uniquely determines the two fold marginals  $\ga^{xy} = \ga_x^y \ot \mu$, $\ga^{xz} = \ga_x^z \ot \mu$, and so the only degrees of freedom are in the conditional probabilities $\ga_x^{yz}$ of $\ga = \ga_x^{yz} \ot \mu$, which must all be martingale couplings of $\ga_x^y$ and $\ga_x^z$. We compute the marginal couplings of these conditional probabilities withe minimal entropy and use these values in Table \ref{threeperiodMartOTDataTable}.

In Figure \ref{threePerMartPlotODEvsSK}, we compare the optimal measures calculated by the ODE and Sinkhorn algorithms. The top row of sub-graphs represents 2D projections of pairs of marginals obtained using the ODE method. The bottom row shows sub-graphs generated using the Sinkhorn algorithm. Theoretically, both the $xy$ and $xz$ projections are left monotone couplings for these particular marginals (at least in the unregularized limit, by the Lemma 6.5 of \cite{NutzStebeggTan20} and the definition of multi-period left-monotone coupling), and so the conditional probabilities $\gamma_{x}^y$ and $\gamma_x^z$ are supported on two disjoint consecutive intervals for each $x$. Although the $xy$ projection looks a little bit spread for the ODE generated measure, we encounter a similar situation for the measure generated by Sinkhorn algorithm. Nevertheless, as shown in the table \ref{threeperiodMartOTDataTable}, the values generated by both methods still fall in the desired range. 

\begin{figure}[h!]
    \centering
    \includegraphics[width=1\linewidth]{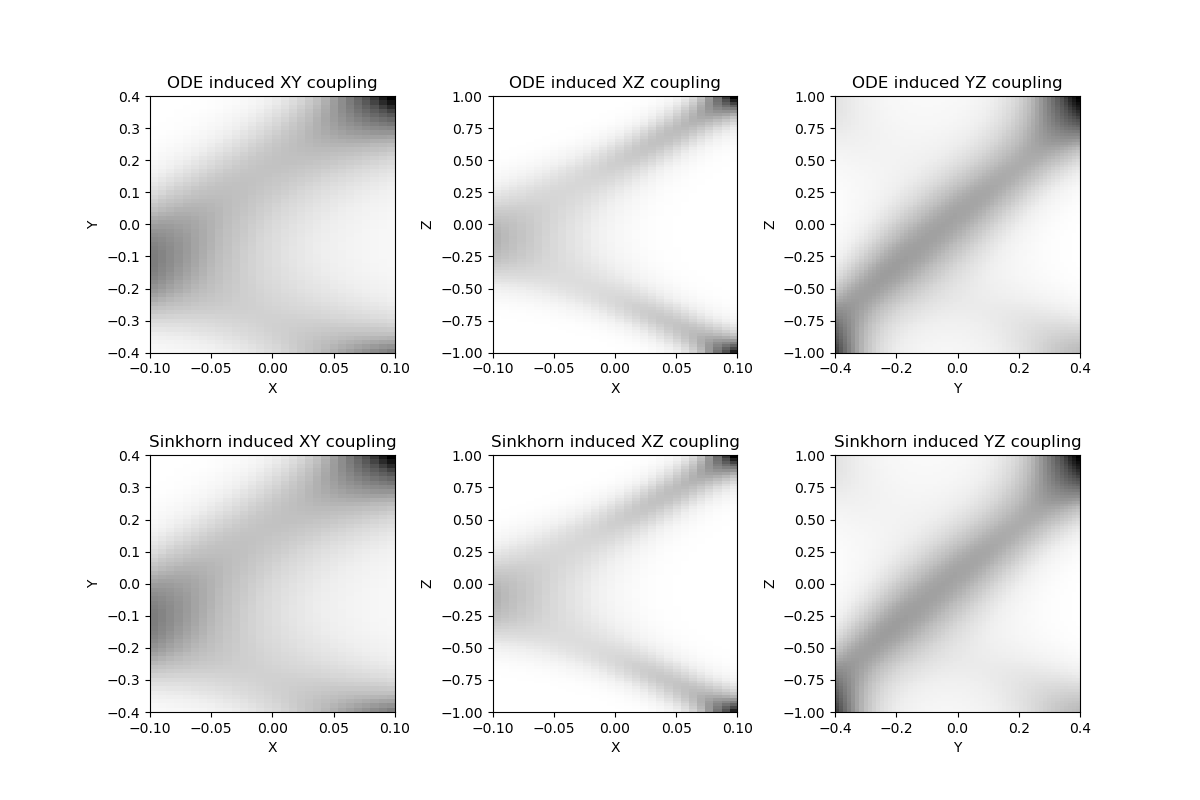}
    \caption{2D projections for pairs of marginals of the optimal measures for multi-period martingale transport problem. Top: generated by ODE method. Bottom: generated by Sinkhorn algorithm}
    \label{threePerMartPlotODEvsSK}
\end{figure}

\begin{table}[h!]
    \centering
    \begin{tabular}{|c|c|c|}
        \hline
                       & ODE & Sinkhorn \\ \hline
        Computed regularized primal value  & 0.3807  & 0.3807\\ \hline
        Optimal unregularized primal value   & 0.3767  & 0.3767\\ \hline
        Optimal unregularized primal value + Entropy  & 0.4127  & 0.4127 \\ \hline
        Iterations     & 25      & 60   \\ \hline
        CPU time (sec) & 8.71    & 7.40   \\ \hline
    \end{tabular}
    \caption{Comparison of the performance between 4-th order Runge-Kutta ODE method and Sinkhorn algorithm for multi-period martingale transport problem.}
    \label{threeperiodMartOTDataTable}
\end{table}
\subsection{Wasserstein geodesics and barycenters}\label{subsect: barycenters and geodesics}
 
 A slight variant of the theory developed in this paper can be used to compute, and provide a new perspective on, Wasserstein geodesics (also known as displacement interpolants \cite{McCann1994}) and barycenters \cite{aguehcarlier2011}. Let us consider, firstly, the computation of a geodesic (since we are considering a regularized problem, it would be more appropriate to refer to it as the \emph{entropic interpolant}) between two measures $\mu^1$ and $\mu^2$. The geodesic at time $\ep$ can be obtained as the second marginal of a coupling $\gamma\in\Prob(X^1\times Z\times X^2)$ whose first and third marginals are $\mu^1$ and $\mu^2$, respectively, and such that it minimizes an optimal transport problem with cost
\[ c_\ep(x^1,z,x^2)=(1-\ep)|x^1-z|^2+\ep|z-x^2|^2.\]

After regularizing, we minimize
$$
\int_{X^1 \times Z \times X^2}c_\ep d\gamma +\eta H_{\mu^1 \otimes \mathcal{U} \otimes \mathcal \mu^2}(\gamma).
$$\footnote{Note that here, since the $z$ marginal of $\gamma$ is not prescribed in the problem, we take the  second marginal of the reference measure to be uniform $\mathcal{U}$.}
Note that our formulation of the geodesic problem actually falls slightly outside the general ODE framework developed in Section \ref{sect: ODE} in two ways.  First, the problem is not a typical optimal transport problem, since the marginal of the variable $z$ is not prescribed (this, the displacement interpolant, is in fact what we are trying to find).  Secondly, the cost takes the form $c_\ep = c_0+\ep c$, rather than the form $c_\ep =\ep c$ dealt with above.  Nonetheless, the development of the ODE \eqref{discreteOTExtraLinConsEntODE} and proof of its well posedness follow almost exactly as above, and so we can solve the ODE to approximate the displacement interpolant.  Notice that now the parameter $\ep$ plays the role of the time. In Figure \ref{fig:entropic_geo} we plot the interpolant, between the uniform on $[0,1]$ and the sum of two gaussians, computed by using the ODE approach. 
\begin{figure}[h!]
        \TabFive{
        \includegraphics[width=.19\linewidth]{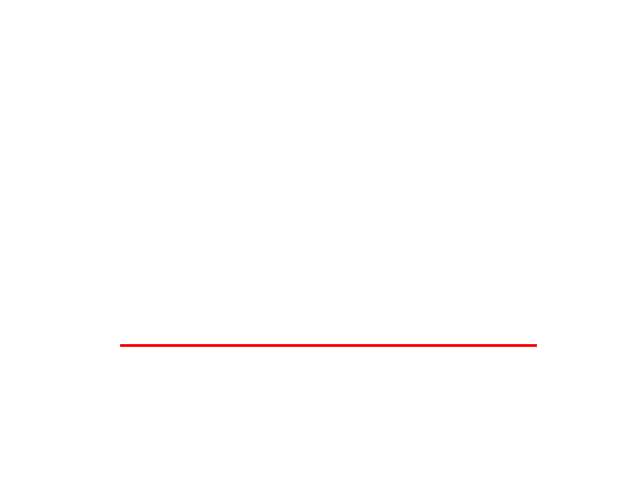}&
        \includegraphics[width=.19\linewidth]{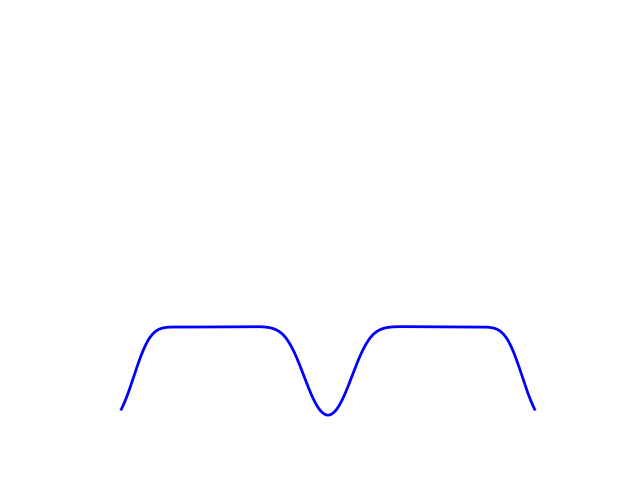}&
        \includegraphics[width=.19\linewidth]{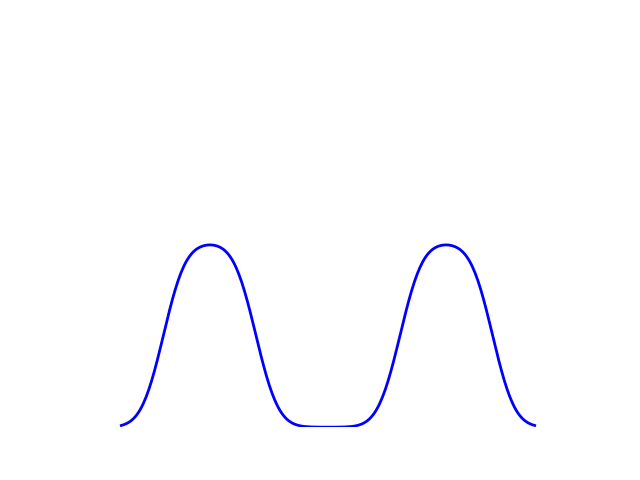}&
        \includegraphics[width=.19\linewidth]{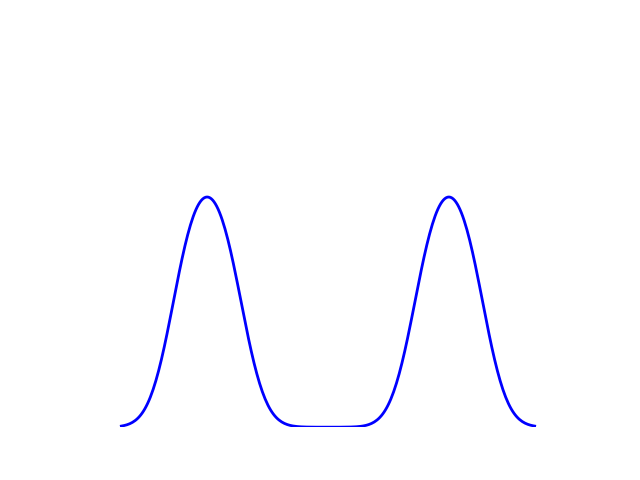}&
        \includegraphics[width=.19\linewidth]{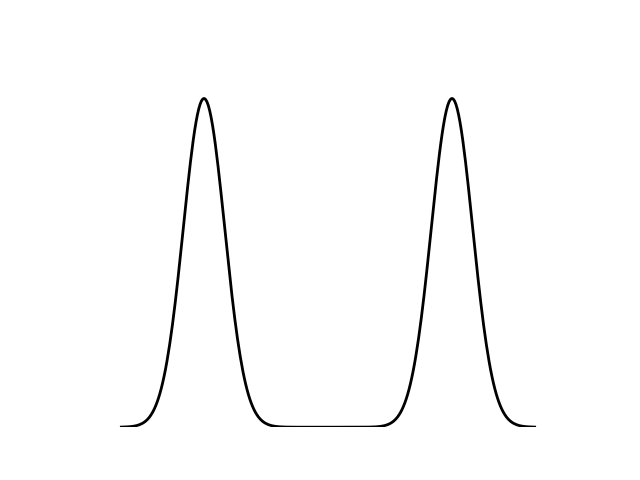}
        }
        \caption{Entropic interpolant at different time steps $\ep$}
        \label{fig:entropic_geo}
        \end{figure}
In a very similar way one can extend this approach to the computation of Wasserstein barycenter \cite{aguehcarlier2011}: in this case we look for a coupling $\gamma\in\Prob(\times_{i=1}^n X^i\times Z)$ having $n$ fixed marginals equal to $\mu^i$ such that it minimizes the optimal transport problem with cost
\[c(x^1,\ldots,x^n,z)=\sum_{i=1}^n\lambda_i(\ep)|x^i-z|^2,\]
where $\lambda_i(\ep)$ are the usual weights such that $\lambda_i(\ep)\geq 0$ and $\sum_i\lambda_i(\ep)=1$ for every $\ep$. In this case the ODE returns the barycenter for every weights $\lambda_i(\ep)$ as they vary in $\ep$.
\begin{remark}[Initial condition]
One difference with the theory we have developed above concerns the initial condition for the Cauchy problem: here, one has to solve a simplified entropic optimal transport problem where the cost is not zero. For instance, in the geodesic case, one should solve a three marginal  problem with cost $|x^1-z|^2$. In particular, by the optimality conditions, the optimal dual variable $u$ associated to the marginal $\mu^1$ has the following explicit form
\begin{equation}\label{eqn: potential for geodesics}
u_r=- \eta \log\lt[\sum_{s} Z\et\lt[\foe ( - \ep c_{rs}) \rt] \rt],
\end{equation}
where $c_{rs}$ is the cost $c_\ep$ above evaluated at $\ep=0$. Notice that since the potential associated to the  marginal $\mu^2$ is a constant we can easily set it to be equal to $0$ (or, equivalently, $Z=1$).

As for the barycenter problem, choose $\lambda_i$ such that at $\ep=0$ there is only one weight equals to 
$1$, so that the curve starts at one of the marginal measures; say $\mu^1$.  The cost function then becomes $|x^1-z|^2$, and the initial potentials can be determined similarly to above; that is, all potentials are equal to $0$ except for the one corresponding to $x^1$, which is determined again by formula \eqref{eqn: potential for geodesics}above with $Z=1$.
\end{remark}

\smallskip

\noindent{\textbf{Acknowledgments.}} L.N. is partially on academic leave at Inria (team Matherials) for the year 2023-2024 and acknowledges the hospitality if this institution during this period. His work  benefited from the support of the FMJH Program PGMO,  from H-Code, Université Paris-Saclay and from the ANR project GOTA (ANR-23-CE46-0001).
 The work of B.P. was partially supported by  National Sciences and Engineering Research Council of Canada Discovery Grant number  04658-2018.  The work of J.H. was completed in partial fulfillment of the requirements of a doctoral degree in applied mathematics at the University of Alberta.
\appendix

\section{Convergence of  entropic optimal transport with extra linear constraints to the unregularized limit}\label{AppPrimeConvProof}

We aim to show that the optimal measure for the entropic optimal transport problem with extra linear constraints will converge to an optimal solution of the unregularized problem that has the minimum entropy relative to the product measure. 
\begin{lemma}\label{LinCtrSetCompact}
    $\Pi^Q(\mu^1, ..., \mu^n)$ is compact in the weak topology. 
\end{lemma}
\begin{proof}
Since $\Pi(\mu^1, ..., \mu^n)$ is a weakly compact set, we only need to show that $\Pi^Q$ is a closed subset of $\Pi$.
If $\ga_k \to \ga_\infty \in \Pi$ weakly, then
\be
    \int_X q d\ga_\infty = \lim_{k \to \infty} \int_X q d\ga_k = 0, \qquad \forall q \in Q \nn
\ee
since $Q \subset C_b(X)$. Hence $\ga_\infty \in \Pi^Q$.
\end{proof}
    
Given a lower semi-continuous cost $c$. Let $\eta_k \downarrow 0$ and $\ga_{\eta_k}$ solve:
\be \label{eqn: regularized minimizers}
    \inf_{\ga \in \Pi^Q(\mu^1,...,\mu^n)} \int_X c d\ga + \eta_k H_{\ot_{i=1}^n \mu^i} (\ga).
\ee
Note that for general marginals, it is often the case that all minimizers of \eqref{eqn: regularized minimizers} have infinite entropy.  To make sense of the statement that cluster points minimize the entropy among solutions to the unregularized problem, we therefore require an additional assumption (uniform boundedness of the entropy in the Lemma below).  This assumption is certainly satisfied in the discrete case, which is the main focus of this paper.
\begin{lemma}\label{lem: convg min ent sol}
Suppose that $\eta_k$ is a sequence converging to $0$ such that the minimizers $\ga_{\eta_k}$ in \eqref{eqn: regularized minimizers} converge weakly to some $\ga_0 = \lim \ga_{\eta_k}$, and that the entropy of the $\ga_{\eta_k}$ is uniformly bounded, that is, $\exists M > 0$ such that $\forall k, H_{\ot_{i=1}^n \mu^i} (\ga_{\eta_k}) < M$.
Then $\ga_0$ solves the unregularized optimal transport problem with additional linear constraints \eqref{eqn: unregularized primal}.  Moreover, $\gamma_0$ has minimal entropy among all solutions of \eqref{eqn: unregularized primal}:
$$
\gamma_0\in argmin_{\gamma \text{ solves } \eqref{eqn: unregularized primal}} H_{\ot_{i=1}^n \mu^i}(\gamma).
$$

\end{lemma}
\begin{proof}
By the assumption of $\ga_{\eta_k}$, for any measure $\ga \in \Pi^Q$ such that $H_{\ot_{i=1}^n \mu^i}(\ga) < \infty$, we have
\be
    \int_X c d\ga_{\eta_k} + \eta_k H_{\ot_{i=1}^n \mu^i} (\ga_{\eta_k}) \leq \int_X c d\ga + \eta_k H_{\ot_{i=1}^n \mu^i} (\ga). \nn
\ee
By taking the limit of both side for $k$, with the lower semi-continuity of the cost and uniformly boundedness of the relative entropy of $\ga_{\eta_k}$, we thus have:
\be
    \int_X c d\ga_0 \leq \int_X c d\ga. \nn
\ee
Note that by Lemma \ref{LinCtrSetCompact}, $\ga_0 \in \Pi^Q$. Therefore $\ga_0$ solve the original optimal transport problem with extra linear constraints.  

Now, suppose $\bar \ga \in \Pi^Q$ is another optimal measure that solves the unregularized problem with $H_{\ot_{i=1}^n \mu^i} (\bar \ga) < \infty$. Again by the optimality assumption of $\ga_{\eta_k}$,
\be\label{optmConvgeProofIneq1}
    \int_X c d\ga_{\eta_k} + \eta_k H_{\ot_{i=1}^n \mu^i} (\ga_{\eta_k}) \leq \int_X c d\bar \ga + \eta_k H_{\ot_{i=1}^n \mu^i} (\bar \ga).
\ee
The optimality assumption for $\bar \ga$ to the unregularized problem implies
\be\label{optmConvgeProofIneq2}
    \int_X c d\ga_{\eta_k} \geq \int_X c d\bar \ga. 
\ee
Combining \eqref{optmConvgeProofIneq1} and \eqref{optmConvgeProofIneq2} we get
\be
    H_{\ot_{i=1}^n \mu^i} (\ga_{\eta_k}) \leq H_{\ot_{i=1}^n \mu^i} (\bar \ga). \nn
\ee
Taking the limit on both sides again and by the lower semi-continuity of the relative entropy functional \cite{Posner1975}, we deduce that
\be
    H_{\ot_{i=1}^n \mu^i} (\ga_0) \leq H_{\ot_{i=1}^n \mu^i} (\bar \ga). \nn
\ee
This means $\ga_0$ has the minimum relative entropy  among all optimal measures.
\end{proof}

In the following, we prove the convergence of the dual potentials on discrete domain. We adapt to the simplified notation to represent the entropic optimal transport problem with extra linear constraints introduced in section \ref{OTexLinCtrLP}:
\be\label{AppOTexLinCtrLP}
    \min  -b^T v + \eta \sum_{\ell=1}^m \et\lt(\f{A_\ell v - \ep c_\ell}{\eta} \rt) \boldsymbol{\mu}_\ell,
\ee
where the matrix $A$ has full column rank. We assume the set of solutions to the dual unregularized optimal transport problem with extra linear constraints is non-empty and bounded in the sense that if 
\be\label{eq: unregularized discrete dual problem}
    D_0 = \inf  -b^T v 
\ee
subject to $A_\ell v \leq c_\ell$ and
\be
    V_0 \coloneqq \lt\{ v \in \R^d | A_\ell v \leq c_\ell, -b^T v = D_0 \rt\}, \nn
\ee
then there exist an $M > 0$ such that $||v|| < M, \forall v \in V_0$. Moreover we assume that $\forall \ell, c_\ell \geq 0$. Let $v_\eta$ be the optimal solution to the \eqref{AppOTexLinCtrLP} with parameter $\eta$ and $\eta_k \downarrow 0$. We denote $v_k = v_{\eta_k}$ for simplicity.
\begin{proposition}
    The sequence $\{v_k\}_k$ is bounded.  If it converges to some $v_0 = \lim v_k$, then $v_0$ is optimal in \eqref{eq: unregularized discrete dual problem}.
\end{proposition}

\begin{proof}
     Let $\bar v \in V_0$. Then for all $k$,
    \be\label{AppOTexLinCtrDualUpperBdd}
        -b^T v_k + \eta_k \sum_{\ell=1}^m \et\lt(\f{A_\ell v_k - c_\ell}{\eta_k} \rt) \boldsymbol{\mu}_\ell \leq -b^T \bar v + \eta_k \sum_{\ell=1}^m \et\lt(\f{A_\ell \bar v - c_\ell}{\eta_k} \rt) \boldsymbol{\mu}_\ell \leq -b^T \bar v + \eta_k \leq  \bar M
    \ee
    for some $\bar M > 0$. The second inequality comes from the fact that $A_\ell \bar v \leq c_\ell$ and hence the exponential term is less than 1. Since exponential function and $\boldsymbol{\mu}_\ell$ are always positive, from \eqref{AppOTexLinCtrDualUpperBdd} we obtain for all $\ell$:
    \be
        A_\ell v_k \leq \eta_k \log \lt[\f{1}{\eta_k} (\bar M + b^T v_k) \rt] + c_\ell. \nn
    \ee
    If $\{v_k\}_k$ is unbounded and $||v_k|| \to \infty$ as $k \to \infty$ and without loss of generality if $v_k / ||v_k|| \to \hat v \ne 0$, then
    \be
        A_\ell \f{v_k}{||v_k||} \leq \f{\eta_k}{||v_k||} \log \lt[\f{1}{\eta_k} (\bar M + b^T v_k) \rt]  + \f{c_\ell}{||v_k||} \to 0. \nn 
    \ee
    We can conclude that $A \hat v \leq 0$. Moreover, the positive of exponential function and $\mu_i$ also implies that $-b^T \hat u \leq 0$ by the same argument. Therefore, $\bar u + \lambda \hat u$ is also an optimal solution to the unregularized problem which violate the boundedness of $V_0$.

    Assuming $v_k \to v_0$ as $k \to \infty$. Since $\{v_k\}_k$ is bounded, we can find another constant $\hat M > 0$ such that
    \begin{align*}
        \eta_k \sum_\ell \et\lt(\f{A_\ell v_k - \ep c_\ell}{\eta_k} \rt)  \boldsymbol{\mu}_\ell &\leq \hat M\\
        A_\ell v_k &\leq \eta_k \log \lt(\f{\hat M}{\eta_k} \rt) + c_\ell.
    \end{align*}
    By taking $k \to \infty$ for both sides, we deduce that $A_i u_0 \leq c_i$. Again, by \eqref{AppOTexLinCtrDualUpperBdd}:
    \be
        -b^T v_0 \leq \limsup_{k \to \infty} -b^T v_k + \eta_k \sum_{\ell=1}^m \et\lt(\f{A_\ell v_k - c_\ell}{\eta_k} \rt)  \boldsymbol{\mu}_\ell  \leq -b^T \bar v. \nn
    \ee
    Therefore, we can conclude that $v_0$ is also an optimal solution to the unregularized problem.    
    \end{proof}

\section{Calculation of the second derivatives of $C(\ep)$}\label{AppC2ndDerCal}
Recall:
\be
    C''(\ep) = \lt[ \grad_u\f{\p}{\p \ep}\Phi(u(\ep), \ep)\rt]^T \lt[ -D^2_{uu} \Phi(u(\ep), \ep) \rt]^{-1}\lt[ \grad_u\f{\p}{\p \ep}\Phi(u(\ep), \ep)\rt] + \f{\p^2}{\p \ep^2}\Phi(u(\ep), \ep). \label{AppOT2ndorder}
\ee
Let
\be
    e_{rs} = \et\lt[\foe (u_r - \ep c_{rs}) \rt] \mu_r \nn
\ee
and $e_{rs}^c = e_{rs}  c_{rs}, e_{rs}^{c2} = e_{rs} c_{rs}^2 $. It is easy to deduce all the second derivative of $\Phi$:
\begin{align}
    \f{\p^2}{\p \ep^2}\Phi(u(\ep), \ep) &= \foe \sum_s \lt[\f{\sum_r e_{rs}^{c2}}{\sum_r e_{rs}} - \f{\sum_r e_{rs}^c \cdot \sum_r e_{rs}^c}{(\sum_r e_{rs})^2} \rt] \nu_s \nn \\
    \lt(\grad_u \f{\p}{\p \ep}\Phi(u(\ep), \ep)\rt)_i &= -\foe \sum_s  \lt[ \f{e_{is}^c}{\sum_r e_{rs}} - \f{e_{is} \cdot \sum_r e_{rs}^c}{(\sum_r e_{rs})^2} \rt]\ \nu_s \nn\\
    \biggl(D^2_{uu} \Phi(u(\ep), \ep)\biggr)_{i_1i_2} &= \foe \sum_s \lt[ \f{e_{i_1s}}{\sum_r e_{rs}} \delta_{i_1 i_2} - \f{e_{i_1s} \cdot e_{i_2s}}{(\sum_r e_{rs})^2} \rt] \nu_s, \nn
\end{align}
where $\delta_{ip}$ is the Kronecker delta.

Let us calculate the \eqref{AppOT2ndorder} separately. For the second part, since $e_{ij} = \mu_i$ when $\ep = 0$, then:
\be\label{AppOTcost2ndder1}
    \f{\p^2}{\p \ep^2}\Phi(u(0), 0) = \foe \lt( \E[c^2(X, Y)] - \E\lt[(\E[c(X, Y)| Y])^2\rt] \rt).
\ee

Let's deal with the first part of \eqref{AppOT2ndorder}. With the assumption that $u_0 \equiv 0$, then when $\ep = 0$, for every $i \neq 0$
\be
    \lt(\grad_u \f{\p}{\p \ep}\Phi(u(0), 0)\rt)_i = \f{\mu_i}{\eta} \bigl( \E[c(X, Y)] -  \E[c(X, Y)| X = x_i]\bigr) \nn
\ee
and
\be
    -D^2_{uu} \Phi(u(0), 0) = \foe \bigl(-\text{diag}(\bar \mu) + \bar \mu {\bar \mu}^T \bigr), \nn
\ee
where $\bar \mu = \{\mu_i\}_{i \neq 0}$. Note that, by Sherman-Morrison formula, we have:
\begin{align*}
    \lt[-D^2_{uu} \Phi(u(0), 0) \rt]^{-1} &= \eta \lt(-\text{diag}^{-1}(\bar \mu) - \f{\text{diag}^{-1}(\bar \mu)\bar \mu {\bar \mu}^T \text{diag}^{-1}(\bar \mu)}{1 - {\bar \mu}^T \text{diag}^{-1}(\bar \mu)\bar \mu} \rt)\\
    &= \eta \lt(-\text{diag}^{-1}(\bar \mu) - \f{\mathbbm{1} \mathbbm{1}^T}{\mu_0} \rt),
\end{align*}
where $\mathbbm{1}$ is a vector with all entries 1 and $\text{diag}^{-1}(\bar \mu)$ is a diagonal matrix where each diagonal entry is the inverse of $\mu_i$. Therefore
\begin{align*}
     &\lt[ -D^2_{uu} \Phi(u(0), 0) \rt]^{-1}\lt[ \grad_u\f{\p}{\p \ep}\Phi(u(0), 0)\rt] \\
     =& \eta \lt(-\text{diat}^{-1}(\bar \mu) - \f{\mathbbm{1} \mathbbm{1}^T}{\mu_0} \rt) \lt(  \f{\mu_i}{\eta} \bigl( \E[c(X, Y)] -  \E[c(X, Y)| X = x_i]\bigr) \rt)_{i \neq 0}\\
     =& \biggl(\E[c(X, Y)| X = x_i] - \E[c(X, Y)] \biggr)_{i \neq 0}\\ 
     & \qquad \qquad + \lt(\f{1}{\mu_0} \sum_{n \neq 0} \lt[ \E[c(X, Y)| X = x_i] - \E[c(X, Y)]\rt]  \rt)_{i \neq 0}\\
     =& \bigl( \E[c(X, Y)| X = x_i] - \E[c(X, Y)] + \E[c(X, Y)] -\E[c(X, Y)| X = x_0]\bigr)_{i \neq 0}\\
     =& \bigl( \E[c(X, Y)| X = x_i] -\E[c(X, Y)| X = x_0]\bigr)_{i \neq 0}
\end{align*}
and
\begin{align}
    &\lt[ \grad_u\f{\p}{\p \ep}\Phi(u(\ep), \ep)\rt]^T \lt[ -D^2_{uu} \Phi(u(\ep), \ep) \rt]^{-1}\lt[ \grad_u\f{\p}{\p \ep}\Phi(u(\ep), \ep)\rt] \nn\\
    =& \lt(  \f{\mu_i}{\eta} \bigl( \E[c(X, Y)] -  \E[c(X, Y)| X = x_i]\bigr) \rt)_{i \neq 0}^T \nn\\
    & \qquad \qquad \cdot \bigl( \E[c(X, Y)| X = x_i] -\E[c(X, Y)| X = x_0]\bigr)_{i \neq 0} \nn\\
    =& \foe \sum_{i \neq 0} \bigl[ \E[c(X, Y)]\E[c(X, Y)| X = x_i] - \E[c(X, Y)]\E[c(X, Y)| X = x_0] \nn\\
    &\qquad - \lt( \E[c(X, Y)| X = x_i] \rt)^2 + \E[c(X, Y)| X = x_0]\E[c(X, Y)| X = x_i]\bigr] \nn\\
    =& \foe \biggl\{ \E[c(X, Y)]\sum_i \E[c(X, Y)| X = x_i] \mu_i - \E[c(X, Y)]\sum_i \E[c(X, Y)| X = x_0] \mu_i \nn\\   
    & \qquad - \sum_i \lt(\E[c(X, Y)| X = x_i] \rt)^2 \mu_i + \E[c(X, Y)| X = x_0] \sum_i \E[c(X, Y)| X = x_i]\mu_i\biggr\} \nn\\
    & \qquad - \foe\biggl\{ \E[c(X, Y)]\E[c(X, Y)| X = x_0] - \E[c(X, Y)]\E[c(X, Y)| X = x_0] \nn\\
    & \qquad \qquad - \lt( \E[c(X, Y)| X = x_0] \rt)^2 + \E[c(X, Y)| X = x_0]\E[c(X, Y)| X = x_0] \biggr\} \mu_0 \nn\\
    =&\foe \biggl( \lt(\E[c(X, Y)] \rt)^2 - \E\lt[(\E[c(X, Y)| X])^2\rt]\biggr). \label{AppOTcost2ndder2}
\end{align}
Hence, summing \eqref{AppOTcost2ndder1} and \eqref{AppOTcost2ndder2} we get:
\be
    C''(0) = \foe \biggl( \lt(\E[c(X, Y)] \rt)^2 + \E[c^2(X, Y)] - \E \lt[(\E[c(X, Y)| X])^2\rt] - \E \lt[(\E[c(X, Y)| Y])^2\rt]\biggr). \nn
\ee

\subsection{The second order derivatives of $\Phi(u, v,\ep)$ for multi marginal optimal transport}\label{AppMMargOT2ndDerSec}

Recall:
\be
    \Phi(u, v, \ep) = -\sum_r u_r \mu_r - \sum_s v_s \tht_s + \eta \sum_t \log \lt[\sum_{r,s} \et \lt[\foe (u_r + v_s - \ep c_{rst}) \rt] \mu_r \tht_s \rt] \nu_{rst} - \eta. \nn
\ee

Let us denote
\be
    e_{ijk} = \et \lt[\foe (u_i + v_j - \ep c_{ijk}) \rt], \qquad e_{ijk}^c = \et \lt[\foe (u_i + v_j - \ep c_{ijk}) \rt] c_{ijk}.\nn
\ee
The the Hessian matrix of $\Phi$ will be the symmetric matrix:
\be
   D^2 \Phi = 
    \begin{bmatrix}
        D^2_{uu} \Phi & D^2_{uv} \Phi \\
        D^2_{vu} \Phi & D^2_{vv} \Phi \\
    \end{bmatrix}, \nn
\ee
where:
\begin{align*}
    \biggl(D^2_{uu} \Phi(u(\ep), v(\ep), \ep)\biggr)_{i_1i_2} &= \foe \sum_t \lt[ \f{\sum_s e_{i_1st}}{\sum_{r,s} e_{rst}} \delta_{i_1i_2} - \f{\sum_s e_{i_1st} \cdot \sum_s e_{i_2st}}{(\sum_{r,s}e_{rst})^2} \rt] \nu_t\\
    \biggl(D^2_{vv} \Phi(u(\ep), v(\ep), \ep)\biggr)_{j_1j_2} &= \foe \sum_t \lt[ \f{\sum_r e_{rjt}}{\sum_{r,s}e_{rst}} \delta_{j_1j_2} - \f{\sum_r e_{rj_1t} \cdot \sum_r e_{rj_2t}} {(\sum_{r,s} e_{rst})^2} \rt] \nu_t\\
    \biggl(D^2_{uv} \Phi(u(\ep), v(\ep), \ep)\biggr)_{ij} &= \foe \sum_t \lt[ \f{e_{ijt}}{\sum_{r,s}e_{rst}} - \f{\sum_s e_{ist} \cdot \sum_r e_{rjt} }{(\sum_{r,s} e_{rst})^2} \rt]\nu_t
\end{align*}
and the mixed derivative is:
\be
    \grad \f{\p}{\p \ep}\Phi = 
    \begin{bmatrix}
        \grad_u \f{\p}{\p \ep}\Phi\\
        \grad_v \f{\p}{\p \ep}\Phi \\
    \end{bmatrix},
\ee
where:
\begin{align*}
    \lt(\grad_u \f{\p}{\p \ep}\Phi(u(\ep), v(\ep), \ep)\rt)_i &= -\foe \sum_t  \lt[ \f{\sum_s e_{ist}^c}{\sum_{r,s} e_{rst}} - \f{\sum_s e_{ist} \cdot \sum_{r,s} e_{rst}^c}{(\sum_{r,s} e_{rst})^2} \rt]\nu_t\\
    \lt(\grad_v \f{\p}{\p \ep}\Phi(u(\ep), v(\ep), \ep)\rt)_j &= - \foe \sum_t \biggl[\f{\sum_r e_{rjt}^c}{\sum_{r,s} e_{rst}} - \f{\sum_r e_{rjt} \cdot \sum_{r,s} e_{rst}^{c}}{(\sum_{r,s} e_{rst})^2} \biggr] \nu_\ell.
\end{align*}

\subsection{The second order derivatives of $\Phi(v, g,\ep)$ for martingale optimal transport}\label{AppMOT2ndDerSec}

Recall:
\be
    \Phi(v, g, \ep) \coloneqq -\sum_m v_m \nu_m + \eta \sum_n \log\lt[ \sum_m \et \lt[ \foe(v_m + g_n(y_m - x_n) - \ep c_{nm}) \rt] \nu_m \rt]\mu_n - \eta. \nn
\ee

Let denote 
\be
    e_{ij} = \et \lt[\foe(v_j + h_i(y_j - x_i) - \ep c_{ij})\rt] \nu_j \nn
\ee
and $e_{ij}^{xy} = e_{ij}(y_j - x_i), e_{ij}^{xy2} = e_{ij}(y_j - x_i)^2, e_{ij}^{c} = e_{ij} c_{ij}$ and $e_{ij}^{xyc} = e_{ij}(y_j - x_i)c_{ij}$

The Hessian matrix for $\Phi$ will be the symmetric matrix:
\be
    D^2 \Phi = 
    \begin{bmatrix}
        D^2_{vv} \Phi & D^2_{vh} \Phi \\
        D^2_{hv} \Phi & D^2_{hh} \Phi \\
    \end{bmatrix}, \nn
\ee
where:
\begin{align*}
    \biggl(D^2_{vv} \Phi(v(\ep), h(\ep), \ep)\biggr)_{j_1 j_2} &= \foe \sum_r \lt[ \f{e_{rj}}{\sum_s e_{rs}} \delta_{j_1 j_2} - \f{e_{rj_1} \cdot e_{rj_2}}{(\sum_s e_{rs})^2} \rt] \mu_r\\
    \biggl(D^2_{hh} \Phi(v(\ep), h(\ep), \ep)\biggr)_{i_1 i_2} &= \foe \lt[ \f{\sum_s e_{i_1s}^{xy2}}{\sum_s e_{i_1s}} - \f{\sum_s e_{i_1s}^{xy} \cdot \sum_s e_{i_2s}^{xy}}{(\sum_s e_{i_1s})^2} \rt] \mu_{i_1} \delta_{i_1 i_2}\\
    \biggl(D^2_{vh} \Phi(v(\ep), h(\ep), \ep)\biggr)_{ij} &= \foe \lt[ \f{e_{ij}}{\sum_s e_{is}} - \f{e_{ij} \cdot \sum_s e_{is}^{xy} }{(\sum_s e_{is})^2} \rt]\mu_i,
\end{align*}
where $\delta_{i \ell}$ is the Kronecker delta. While the mixed derivative:
\be
    \grad \f{\p}{\p \ep}\Phi = 
    \begin{bmatrix}
        \grad_v \f{\p}{\p \ep}\Phi\\
        \grad_h \f{\p}{\p \ep}\Phi \\
    \end{bmatrix}, \nn
\ee
where:
\begin{align*}
    \lt(\grad_v \f{\p}{\p \ep}\Phi(v(\ep), h(\ep), \ep)\rt)_j &= -\foe \sum_r  \lt[ \f{e_{rj}^c}{\sum_s e_{rs}} - \f{e_{rj} \cdot \sum_s e_{rs}^c}{(\sum_s e_{rs})^2} \rt]\ \nu_j\\
    \lt(\grad_h \f{\p}{\p \ep}\Phi(v(\ep), h(\ep), \ep)\rt)_i &= - \foe \biggl[\f{\sum_s e_{is}^{xyc}}{\sum_s e_{is}} - \f{\sum_s e_{is}^{xy} \cdot \sum_s e_{is}^{c}}{(\sum_s e_{is})^2} \biggr] \mu_i.
\end{align*}

\subsection{The second order derivatives of $\Phi(v, w, g, h,\ep)$ for multi-period optimal transport}\label{AppMMOT2ndDerSec}

Recall:
\begin{align*}
    \Phi(v, w, g, h, \ep) = &- \sum_s v_s \tht_s - \sum_t w_t \nu_t + \eta \sum_r \log \biggl[ \sum_{s,t} \et \biggl[ \foe \bigl( v_s + w_t\\
    &\qquad + g_r(y_s - x_r) + h_{rs}(z_t - y_s) - \ep c_{rst} \bigr) \biggr] \tht_s \nu_t \biggr] \nu_r - \eta. \nn
\end{align*}

By denoting $e_{ijk} = \et[\foe (v_j + w_k + g_i(y_j - x_i) + h_{ij}(z_k - y_j) - \ep c_{ijk})] \tht_j \nu_k$ and $e_{ijk}^{xy} = e_{ijk} (y_j - x_i), e_{ijk}^{yz} = e_{ijk} (z_k - y_j), e_{ijk}^{xy2} = e_{ijk} (y_j - x_i)^2, e_{ijk}^{xyz} = e_{ijk}(y_j - x_i)(z_k - y_j), e_{ijk}^{yz2} = e_{ijk} (z_k - y_j)^2, e_{ijk}^c = e_{ijk} c_{ijk}, e_{ijk}^{xyc} = e_{ijk} (y_j - x_i) c_{ijk}, e_{ijk}^{yzc} = e_{ijk} (z_k - y_j) c_{ijk}, $ 

the Hessian of $\Phi$ will be the symmetric matrix:
\be
    D^2 \Phi = 
    \begin{bmatrix}
        D^2_{vv} \Phi & D^2_{vw} \Phi & D^2_{vg} \Phi & D^2_{vh} \Phi \\
        D^2_{wv} \Phi & D^2_{ww} \Phi & D^2_{wg} \Phi & D^2_{wh} \Phi \\
        D^2_{gv} \Phi & D^2_{gw} \Phi & D^2_{gg} \Phi & D^2_{gh} \Phi \\
        D^2_{hv} \Phi & D^2_{hw} \Phi & D^2_{hg} \Phi & D^2_{hh} \Phi \\
    \end{bmatrix}, \nn
\ee
where
\begin{align*}
    \lt( D^2_{vv} \Phi \rt)_{j_1j_2} &= \foe \sum_r \lt[ \f{\sum_t e_{rj_1t}}{\sum_{s,t} e_{rst}} \delta_{j_1 j_2} - \f{\sum_t e_{rj_1t} \cdot \sum_t e_{rj_2t}}{(\sum_{s,t} e_{rst})^2} \rt] \mu_r\\
    \lt( D^2_{vw} \Phi \rt)_{jk} &= \foe \sum_r \lt[\f{e_{rjk}}{\sum_{s,t} e_{rst}} - \f{\sum_t e_{rjt} \cdot \sum_s e_{rsk}}{(\sum_{s,t} e_{rst})^2} \rt] \mu_r\\
    \lt( D^2_{vg} \Phi \rt)_{ji} &= \foe \lt[\f{\sum_t e_{ijt}}{\sum_{s,t} e_{ist}} - \f{\sum_t e_{ijt} \cdot \sum_{s,t }e_{ist}^{xy}}{(\sum_{s,t} e_{ist})^2} \rt] \mu_i \\
    \lt( D^2_{vh} \Phi \rt)_{j_1,(ij_2)} &= \foe \lt[\f{\sum_t e_{ij_2}^{yz}}{\sum_{s,t} e_{ist}} \delta_{j_1 j_2} - \f{\sum_t e_{ij_1} \cdot \sum_t e_{ij_2}^{yz}}{(\sum_{s,t} e_{ist})^2} \rt] \nu_i\\
    \lt( D^2_{ww} \Phi \rt)_{k_1k_2} &= \foe \sum_r \lt[\f{\sum_s e_{rk_1}}{\sum_{s,t} e_{rst}} \delta_{k_1 k_2} - \f{\sum_s e_{rsk_1} \cdot \sum_s e_{rsk_2}}{(\sum_{s,t} e_{rst})^2} \rt] \mu_r \\
    \lt( D^2_{wg} \Phi \rt)_{ki} &= \foe \lt[ \f{\sum_s e_{isk}^{xy}}{\sum_{s,t} e_{ist}} - \f{\sum_s e_{isk} \cdot \sum_{s,t} e_{ist}^{xy}}{(\sum_{s,t} e_{ist})^2} \rt] \mu_i\\
    \lt( D^2_{wh} \Phi \rt)_{k,(ij)} &= \foe \lt[\f{e_{ijk}^{yz}}{\sum_{s,t} e_{ist}} - \f{\sum_s e_{isk} \cdot \sum_t e_{ijt}^{yz}}{(\sum_{s,t} e_{ist})^2} \rt] \mu_i \\
    \lt( D^2_{gg} \Phi \rt)_{i_1 i_2} &= \foe \lt[\f{\sum_{s,t} e_{i_1st}^{xy2}}{\sum_{s,t} e_{i_1st}} - \f{\sum_{s,t} e_{i_1st}^{xy} \cdot \sum_{s,t} e_{i_2st}^{xy}}{(\sum_{s,t} e_{i_1st})^2} \rt] \mu_{i_1} \delta_{i_1 i_2}\\
    \lt( D^2_{gh} \Phi \rt)_{i_1,(i_2j)} &= \foe \lt[\f{\sum_t e_{i_1jt}^{xyz}}{\sum_{s,t} e_{ist}} - \f{\sum_{s,t} e_{i_1st}^{xy} \cdot \sum_t e_{i_2jt}^{yz}}{(\sum_{s,t} e_{ist})^2} \rt] \mu_{i_1} \delta_{i_1 i_2} \\
    \lt( D^2_{hh} \Phi \rt)_{(i_1 j_1),(i_2 j_2)} &= \foe \lt[ \f{\sum_t e_{i_1j_1t}^{yz2}}{\sum_{s,t} e_{i_1st}} \delta_{j_1 j_2} - \f{\sum_t e_{i_1 j_1 t}^{yz} \cdot \sum_t e_{i_2 j_2 t}^{yz}}{(\sum_{s,t} e_{ist})^2} \rt] \mu_{i_1} \delta_{i_1 i_2}.
\end{align*}
Here we flatten the matrix $\{h_{ij}\}$ into a long vector and abuse the notation that we still call it with the index $(ij)$. For the mixed second order derivatives:
\be
    \grad \f{\p}{\p \ep}\Phi = 
    \begin{bmatrix}
        \grad_v \f{\p}{\p \ep}\Phi \\
        \grad_w \f{\p}{\p \ep}\Phi \\
        \grad_g \f{\p}{\p \ep}\Phi \\
        \grad_h \f{\p}{\p \ep}\Phi \\
    \end{bmatrix}, \nn
\ee
where
\begin{align*}
    \lt(\grad_v \f{\p}{\p \ep}\Phi \rt)_j &= -\foe \sum_r \lt[\f{\sum_t e_{rjt}^{c}}{\sum_{s,t} e_{rst}} - \f{\sum_t e_{rjt} \cdot \sum_{s,t} e_{rst}^c}{(\sum_{s,t} e_{rst})^2} \rt] \mu_r \\
    \lt(\grad_g \f{\p}{\p \ep}\Phi \rt)_i &= -\foe \lt[\f{\sum_{s,t} e_{ist}^{xyc}}{\sum_{s,t} e_{ist}} - \f{\sum_{s,t} e_{ist}^{xy} \cdot \sum_{s,t} e_{ist}^c}{(\sum_{s,t} e_{ist})^2} \rt] \mu_i\\
    \lt(\grad_w \f{\p}{\p \ep}\Phi \rt)_k &= -\foe \sum_r \lt[\f{\sum_s e_{rsk}^c}{\sum_{s,t} e_{rst}} - \f{\sum_s e_{rsk} \cdot \sum_{st} e_{rst}^c}{(\sum_{s,t} e_{rst})^2} \rt] \mu_r \\
    \lt(\grad_h \f{\p}{\p \ep}\Phi \rt)_{(ij)} &= -\foe \lt[\f{\sum_t e_{ijt}^{yzc}}{\sum_{s,t} e_{ist}} - \f{\sum_t e_{ijt}^{yz} \cdot \sum_{s,t} e_{ist}^c}{(\sum_{s,t} e_{ist})^2} \rt] \mu_i.
\end{align*}

\bibliographystyle{plain}
\bibliography{biblio}

\end{document}